\documentclass[11pt]{article}
\usepackage{amsmath,tabu}
\usepackage{amssymb}
\usepackage{amsfonts}
\usepackage{amsthm}
\usepackage{slashed}
\usepackage{cancel}
\usepackage{mathtools}
\usepackage{paralist}
\usepackage{graphicx}
\usepackage{float}
\usepackage{tikz}
\usepackage{pgfplots}
\usepackage[hmargin = 2.5 cm,vmargin = 2.7 cm]{geometry}
\usepackage{array}
\usepackage{hyperref}
\numberwithin{equation}{section}
\usepackage{cleveref}
\usepackage{cite}
\usepackage[affil-it]{authblk}
\usepackage{tikz-cd}
\usepackage{placeins}

\usepgfplotslibrary{fillbetween}
\usetikzlibrary{calc}

\newtheorem{lemma}[equation]{Lemma}
\newtheorem{proposition}[equation]{Proposition}
\newtheorem{theorem}[equation]{Theorem}
\newtheorem{corollary}[equation]{Corollary}
\newtheorem*{mThm*}{Main Theorem}

\theoremstyle{definition}
\newtheorem{definition}[equation]{Definition}
\newtheorem*{definition*}{Definition}

\theoremstyle{remark}
\newtheorem{remark}[equation]{Remark}
\newtheorem*{remark*}{Remark}

\newcommand{\Vol}{\mathrm{Vol}}
\newcommand{\im}{\mathrm{im}}
\newcommand{\ind}{\mathrm{ind}}
\newcommand{\Sp}{\mathrm{Sp}}
\newcommand{\SU}{\mathrm{SU}}

\newcommand{\Id}{\mathrm{Id}}
\newcommand{\SO}{\mathrm{SO}}
\newcommand{\CP}[1]{\mathbb{C}P^{#1}}
\newcommand{\Spin}{\mathrm{Spin}}

\newcommand{\coker}{\mathrm{Coker}}

\newcommand{\ig}{\scriptscriptstyle}

\title{Deformations of asymptotically conical Spin(7)-manifolds}

\author{Fabian Lehmann \vspace{-3mm}}
\affil{Simons Center for Geometry and Physics, Stony Brook University
\\ \vspace*{1mm}
\textup{flehmann@scgp.stonybrook.edu}}
\date{\today}

\begin{document}

\maketitle

\begin{abstract}
We consider the moduli space $\mathcal{M}_{\nu}$ of torsion-free, asymptotically conical (AC) Spin(7)-structures which are defined on the same manifold and asymptotic to the same Spin(7)-cone with decay rate $\nu<0$. We show that $\mathcal{M}_{\nu}$ is an orbifold if $\nu$ is a generic rate in the non-$L^2$ regime $(-4,0)$. 
Infinitesimal deformations are given by topological data and solutions to a non-elliptic first-order PDE system on the compact link of the asymptotic cone. As an application, we show that the classical Bryant--Salamon metric on the bundle of positive spinors on $S^4$ has no continuous deformations as an AC Spin(7)-metric.
\end{abstract}

\section{Introduction}

Spin(7) is one of the exceptional holonomy groups on Berger's list. Riemannian manifolds with holonomy group contained in Spin(7) are Ricci flat and, therefore, of great interest in geometry and physics.
In this article we study deformations of \textit{asymptotically conical} (AC) Spin(7)-manifolds. These are non-compact Spin(7)-manifolds which at infinity converge to a Spin(7)-cone. Thus they are the Spin(7) analogue of
\textit{asymptotically locally Euclidean} (ALE) hyperk\"ahler manifolds in dimension 4 such as the well-known Eguchi--Hanson metric on $T^*S^2$. 
The first complete metric with holonomy Spin(7), which was constructed by Bryant--Salamon \cite{BS} in 1989, is of this asymptotic type. It lives on $\mathbf{S}_+(S^4)$, the bundle of positive spinors on the 4-sphere. Recently, the author has proved the existence of two further examples, 
one on the unique non-trivial rank 3 vector bundle over $S^5$ and one on the universal quotient bundle of $\CP{2}$ \cite{myarticle1}.  

An important feature of the asymptotic geometry of AC Spin(7)-manifolds is the rate of convergence to the asymptotic cone. In the setting of AC manifolds, we are interested in polynomial decay, i.e. where the Spin(7)-structure and the associated metric decay to the Spin(7)-cone like $r^{\nu}$, $\nu<0$, where $r$ is the radial function of the cone. $\nu$ is called the \textit{decay rate}. Tensors which decay with rate $\nu < -4$ are square-integrable. 
Our main result is that the moduli space is an orbifold in the non-$L^2$ regime ($-4 < \nu < 0$), and we derive a formula for its dimension. 

\begin{mThm*}
\label{main-thm}
Let $(M,\psi)$ be an AC Spin(7)-manifold of rate $\nu$. For generic $\nu \in (-4,0)$ the moduli space $\mathcal{M}_{\nu}$ of torsion-free AC Spin(7)-structures
on $M$ asymptotic to the same Spin(7)-cone at the same rate modulo an appropriate notion of equivalence is an orbifold. The dimension of $\mathcal{M}_{\nu}$ is determined by topological data of $M$, and solutions to systems of differential equations on the link of the asymptotic cone.
\end{mThm*}

For a more precise formulation see Theorem \ref{Thm-precise-formulation}.
The deformation theory of compact Spin(7)-manifolds has been developed by Joyce \cite{big-joyce}. The moduli space is always smooth and infinitesimal deformations can be expressed in terms of harmonic forms. Several standard techniques in the compact setting do not carry over to the non-compact setting. For example, we frequently are in situations where integration by parts is not available.
Furthermore, the mapping properties of differential operators behave differently in the non-compact setting as compared to the compact setting. Nordstr\"om \cite{Nordstrom-PhD} developed the deformation theory of exponentially asymptotically cylindrical (EAC) Spin(7)-manifolds using analysis on non-compact manifolds. In contrast to the AC setting, in the EAC setting all decay rates lie in the $L^2$-regime because of the exponential decay, and the dimension of the moduli space only depends on topological data. Our work is most closely related to the deformation theory of $\mathrm{G}_2$-conifolds developed by Karigiannis--Lotay \cite{KL}. In our set-up we consider the moduli space $\mathcal{M}_{\nu}$ of AC manifolds of a particular rate $\nu$. To prove smoothness of an orbifold chart, our strategy is to compute infinitesimal deformations and then use the implicit function theorem adapted to appropriate Banach spaces. The computation of the infinitesimal deformations is carried out in several steps. Analogously to the use of Hodge theory in the compact setting yet more intricate, deformations which lie in $L^2$ can be related to the topology of the manifold $M$. As we vary the rate and enter the non-$L^2$ regime, we use the analysis on weighted Sobolev and H\"older spaces as described by Lockhart--McOwen  \cite{Lockhart-McOwen}. Outside a discrete subset of so-called \textit{critical} rates the relevant differential operators are Fredholm and the space of deformations remains constant. As we cross a critical rate the added deformations can be related to particular solutions of differential equations on the asymptotic cone. In the range of rates considered by us we can formulate these equations purely on the compact link of the cone.
If $\nu < -4$, the above program cannot be carried out: the operator describing the infinitesimal deformations may not be surjective and hence the implicit function theorem cannot be applied. Therefore, deformations of AC Spin(7)-metrics with $\nu < -4$ may be obstructed. This resembles the $\mathrm{G}_2$-setting \cite{KL}.
There are several similar articles concerned with the deformation theory of calibrated submanifolds: Marshall \cite{Marshall} and Pacini \cite{ACSlag} studied deformations of AC Special Lagrangian submanifolds in $\mathbb{C}^m$ and Lotay studied 
deformations of AC coassociative and associative submanifolds of $\mathrm{G}_2$-manifolds, cf. \cite{ACcoasso} and \cite{ACasso}, respectively.

In our moduli problem we only consider torsion-free Spin(7)-structures which are asymptotic to a fixed Spin(7)-cone. Denoting the link of the cone by $\Sigma$, the space of torsion-free, conical Spin(7)-structures on $(0,\infty)\times \Sigma$ corresponds to the space of nearly parallel $\mathrm{G}_2$-structures on $\Sigma$. We do not expect this space to have a ``nice'' structure in general, e.g. that of a manifold. Alexandrov--Semmelmann \cite{alexandrov2012} showed that the homogeneous, nearly parallel $\mathrm{G}_2$-structure on the Aloff--Wallach space $N(1,1)$ has an 8-dimensional space of infinitesimal deformations, but by a recent result of
Dwivedi--Singhal \cite{dwivedi2020deformation} these are all obstructed. 
This picture is similar to the case of $\mathrm{G}_2$-cones,
where Foscolo \cite{NK-defo} showed that deformations of nearly  K\"ahler manifolds in general are obstructed.
This aspect is another difference to the EAC case \cite{Nordstrom-PhD}. The link of a $\mathrm{G}_2$-cylinder is a compact Calabi--Yau manifold, and the link of a Spin(7)-cylinder is a compact $\mathrm{G}_2$-manifold. Deformations of these are well understood, which allows a more inclusive set-up for the moduli space.

That the moduli space is in general an orbifold rather than a manifold is owed to the fact that the stabiliser of a torsion-free AC Spin(7)-structure in the group of diffeomorphisms decaying to the identity
can be non-trivial. While we can exclude any continuous such symmetries,
the stabiliser can still be a non-trivial finite group. If this group does not act trivially on a slice for the diffeomorphism action, we only obtain an orbifold rather than a manifold chart. We present one criterion to check if the stabiliser acts trivially on the orbifold chart. The tangent space of the orbifold chart at a torsion-free AC Spin(7)-structure $\psi$ of rate $\nu$ is related to closed 4-forms which are anti-self-dual with respect to $\psi$ and decay with rate $\nu$. If the projection of these forms to the fourth cohomology group of $M$ is injective, we can conclude that each element in the orbifold chart represents a different point in the  moduli space, i.e. that the quotient by the stabiliser is trivial.

As an application of the main theorem we show that the Bryant--Salamon Spin(7) holonomy metric on $\mathbf{S}_{+}(S^4)$ is locally rigid,
 modulo scaling, as a torsion-free AC Spin(7)-structure on $\mathbf{S}_{+}(S^4)$ asymptotic to the same Spin(7)-cone. We solve the differential equations on the link by following Alexandrov--Semmelmann \cite{alexandrov2012}, who compute infinitesimal Einstein deformations of
 normal homogeneous nearly parallel $\mathrm{G}_2$-manifolds with standard
 invariant metrics. Under these constraints the differential equations can be solved with purely representation theoretic methods.
This example strongly relies on the condition that the homogeneous metric on the link is normal and standard. The recent examples of AC Spin(7)-manifolds from \cite{myarticle1} are asymptotic to a more sophisticated Spin(7)-cone, which does not allow us to carry out similar computations.

An important subclass of Spin(7)-manifolds are Calabi--Yau 4-folds. For a given AC Calabi--Yau 4-fold we can use the main theorem to obtain deformations as a Spin(7)-structure. 
Given that a large number of AC Calabi--Yau 4-folds are known while only very few AC Spin(7) holonomy metrics are known to exist, this leads to the interesting question: can an AC Calabi--Yau metric
on a manifold of real dimension 8 be deformed to an AC metric with holonomy Spin(7)? It is known that infinitesimal deformations of Calabi--Yau structures on compact manifolds are unobstructed (see \cite{goto,tian1987smoothness,todorov1989weil}). By adjusting for example the approach in \cite{goto} to the analytic framework of AC manifolds, it is reasonable to believe that infinitesimal deformations of AC Calabi--Yau structures are unobstructed in an interesting range of decay rates. Therefore, 
if each infinitesimal Spin(7)-deformation is induced by an infinitesimal  $\SU(4)$-deformation, this would be strong evidence to provide a negative answer to the above question. The author has partially persued the question of comparing infinitesimal $\SU(4)$- and Spin(7)-deformations in his PhD thesis \cite{myPhD}, however without any conclusion.

\paragraph*{Acknowledgements} This work is a result of the author's PhD thesis and 
was supported by the Engineering and Physical Sciences Research Council [EP/L015234/1], 
the EPSRC Centre for Doctoral Training in Geometry and Number Theory (The London School of Geometry and Number Theory), University College London.
I want to thank my PhD supervisors Mark Haskins, Jason Lotay and Lorenzo Foscolo for their support, and my PhD examiners Johannes Nordstr\"om and Simon Salamon for their helpful comments.

\section{Preliminaries}

\subsection{Spin(7)-geometry}

\label{Spin(7)-intro}

We give a brief review of Spin(7)-geometry. For more details we refer to
\cite{Salamon-book}, \cite{big-joyce} and \cite{Nordstrom-PhD}.
We first discuss the linear algebraic picture.
The spin representation of Spin(7) has a real form which can be identified with $\mathbb{R}^8$. Under this action $\Spin(7)$ can be characterised as the stabiliser in $\mathrm{GL}(8,\mathbb{R})$ of the 4-form 
\begin{align*}
\psi_0
=&
dx_{1234}+dx_{1256}+dx_{1278}+dx_{1357}-dx_{1368}
-dx_{1458}-dx_{1467}
\\
&-dx_{2358}-dx_{2367}-dx_{2457}
+dx_{2468}+dx_{3456}+dx_{3478}+dx_{5678},
\end{align*}
where $(x_1,\dots,x_8)$ are coordinates on $\mathbb{R}^8$.
The action of Spin(7) on $\mathbb{R}^8$ induces an action on the exterior algebra. We get the following decomposition into irreducible components:
\begin{align*}
\Lambda^2 (\mathbb{R}^8)^*
=
\Lambda^2_{7} \oplus \Lambda^2_{21},
\quad
\Lambda^3 (\mathbb{R}^8)^*
=
\Lambda^3_8 \oplus \Lambda^3_{48},
\quad
\Lambda^4 (\mathbb{R}^8)^*
=
\Lambda^4_1 \oplus \Lambda^4_7 \oplus \Lambda^4_{27} \oplus \Lambda^{4}_{35}.
\end{align*}
Here $\Lambda^k_{q}$ denotes a q-dimensional irreducible subspace.
For higher degree forms we get an analogous decomposition by applying the Hodge star operator. Under the identification $\Lambda^2 (\mathbb{R}^8)^* = \mathfrak{so}(8,\mathbb{R})$ the component $\Lambda^2_{21}$ corresponds to the Lie algebra of Spin(7). $\mathrm{GL}(8,\mathbb{R})$ acts on $\Lambda^4 (\mathbb{R}^8)^*$
by pulling back $\psi_0$. The derivative at the identity gives a map
$\mathfrak{gl}(8,\mathbb{R}) \rightarrow \Lambda^4 (\mathbb{R}^8)^*$.
Under the decomposition 
\begin{align*}
\mathfrak{gl}(8,\mathbb{R})
=
\Lambda^2 (\mathbb{R}^8)^* \oplus S^2 (\mathbb{R}^8)^*
=
\Lambda^2_{7} \oplus \Lambda^2_{21} \oplus \mathbb{R}\, \mathrm{Id} \oplus S^2_0 (\mathbb{R}^8)^*,
\end{align*}
where $S^2_0 (\mathbb{R}^8)^*$ denote the trace-less symmetric bilinear forms on $\mathbb{R}^8$, the kernel corresponds to $\mathrm{Lie}(\mathrm{Spin}(7))=\Lambda^2_{21}$, and $\Lambda^4_1$, $\Lambda^4_7$ and $\Lambda^4_{35}$
are the images of $ \mathbb{R}$, $\Lambda^2_7$ and $S^2_0 (\mathbb{R}^8)^*$, respectively. In particular, the orbit of $\psi_0$ under the action of 
$\mathrm{GL}(8,\mathbb{R})$ has co-dimension 27 
and its tangent space at $\psi_0$ is given by
\begin{align}
\label{tangent-space-orbit}
T_{\psi_0} (\mathrm{GL}(8,\mathbb{R})\cdot\psi_0) = \Lambda^4_1 \oplus \Lambda^4_7 \oplus \Lambda^4_{35}.
\end{align}
$\Lambda^4_{27}$ can be identified with the normal directions at $\psi_0$.

We have the identity $*^2 = \mathrm{Id}$ for the Hodge star operator acting on 4-forms. The induced decomposition in spaces of self-dual and anti-self-dual 4-forms is given by
\begin{align*}
\Lambda^4_+ =\Lambda^4_1 \oplus \Lambda^4_7 \oplus \Lambda^4_{27},
\quad 
\Lambda^4_{-} = \Lambda^{4}_{35}.
\end{align*} 
A further space which will be important to us is $\Lambda^3_8$, which has the description
\begin{align}
\label{3-forms-of-type-8}
\Lambda^3_8 
=
\{
X \lrcorner \psi_0\, |\, X \in \mathbb{R}^8
\}.
\end{align}
The above discussion implies that as Spin(7) representations we have
\begin{align*}
\Lambda^3_8 \cong \mathbb{R}^8, \quad \Lambda^4_7 \cong \Lambda^2_7. 
\end{align*}

Because Spin(7) is simply connected, the inclusion $\mathrm{Spin}(7) \subset \SO(8)$ factors through Spin(8). If we denote the real positive and negative spin representations of Spin(8) by $\sigma_8^+$ and $\sigma_8^-$, respectively, as representations of Spin(7) there are isomorphisms
\begin{align}
\label{LA-spin-reps}
\sigma_8^+ \cong  \Lambda^4_1 \oplus \Lambda^4_7 \cong \mathbb{R} \oplus \Lambda^2_7,
\quad
\sigma_8^- \cong \Lambda^3_8 \cong \mathbb{R}^8.
\end{align}

Now we turn to the global differential geometric picture.
Let $M$ be an oriented 8-manifold. We say that a 4-form $\psi$ is \textit{admissible} if at each point $p\in M$ there is an orientation-preserving isomorphism between $T_pM$ and $\mathbb{R}^8$ which identifies $\psi|_p$ with $\psi_0$. 
We refer to $\psi$ as a \textit{Spin(7)-structure}. $\psi$ reduces the structure group of the frame bundle of $M$ to Spin(7) by considering the subbundle 
$\{u\colon \mathbb{R}^8 \xrightarrow{\sim} T_p M\  |\
u^*\psi_p = \psi_0 \}$. Because Spin(7) is a subgroup of $\SO(8)$, $\psi$ induces in a purely algebraic way a Riemannian metric $g$. 
The condition that the holonomy group of the induced metric is contained in $\Spin(7)$ is equivalent to 
\begin{align}
d\psi = 0.
\label{cond-tf}
\end{align}
In this case we say that $\psi$ is \textit{torsion-free} and $(M,\psi$) a \textit{Spin(7)-manifold}. A key point is that $g$ is Ricci-flat if $\psi$ is torsion-free.

The above decomposition of the exterior algebra into irreducible components gives a global decomposition of the corresponding vector bundles. By abuse of notation, we will denote these subbundles by the same symbols as in the linear picture. This decomposition is preserved by the Hodge Laplacian if the Spin(7)-structure is torsion-free.
We denote the space of all smooth admissible 4-forms on $M$
by $\mathcal{A}(M)$. Just as the orbit of $\psi_0$ is a non-linear subspace of $\Lambda^4 (\mathbb{R}^8)^*$, the space $\mathcal{A}(M)$ is a non-linear subspace of $\Omega^4(M)$. Therefore, the condition \eqref{cond-tf} is non-linear. \eqref{tangent-space-orbit} gives 
\begin{align}
\label{tangent-space-admissible-forms}
T_{\psi} \mathcal{A}(M) = \Gamma(\Lambda^4_1 \oplus \Lambda^4_7 \oplus \Lambda^4_{35}).
\end{align}
An 8-manifold equipped with a Spin(7)-structure is spin because of the inclusion $\mathrm{Spin}(7) \subset \mathrm{Spin}(8)$. If we denote the real positive and negative spin bundle by $\mathbf{S}_+$ and $\mathbf{S}_-$, respectively, by \eqref{LA-spin-reps} there
are isomorphisms of vector bundles 
\begin{align}
\label{spin-bundle-iso}
\mathbf{S}_+ \cong \Lambda^4_1 \oplus \Lambda^4_7, 
\quad
\mathbf{S}_- \cong \Lambda^3_8 \cong TM \cong T^*M.
\end{align}
In particular, if the Spin(7)-structure is torsion-free, the Dirac Laplacian can be identified with the Hodge Laplacian on the respective bundles. 
The identifications \ref{spin-bundle-iso} can be chosen such that the positive and negative Dirac operator correspond to
\begin{subequations}
\begin{gather}
\slashed{D}_{+}:
\Gamma(\Lambda^4_{1\oplus 7}) \rightarrow \Gamma(\Lambda^3_8),
\quad
\gamma \mapsto \pi_8(d^*\gamma),
\label{pos-Dirac}
\\
\slashed{D}_{-}:
\Gamma(\Lambda^3_8)
\rightarrow
\Gamma(\Lambda^4_{1\oplus 7}),
\quad
\gamma
\mapsto
\pi_{1\oplus 7}(d\gamma).
\label{neg-Dirac}
\end{gather}
\end{subequations}

To describe the moduli space of torsion-free Spin(7)-structures on a manifold, 
we need to describe in a systematic way other admissible 4-forms close to a reference Spin(7)-structure $\psi$.
As this can be done fibre-wise, we first return to the local picture in $\mathbb{R}^8$.
The decomposition \eqref{tangent-space-orbit} in tangent and normal directions implies that the derivative of the map
\begin{align*}
(\mathrm{GL}(8,\mathbb{R})\cdot\psi_0) \times \Lambda^4_{27}
\rightarrow \Lambda^4,
\quad
(\psi, \zeta) \mapsto \psi +\zeta, 
\end{align*}
at $(\psi_0,0)$ is an isomorphism. By the inverse function theorem every 4-form sufficiently close to $\psi_0$ can be written in a unique way as the sum of an element in  $(\mathrm{GL}(8,\mathbb{R})\cdot\psi_0)$ and $\Lambda^4_{27}$. In particular,
if $\varepsilon>0$ is chosen sufficiently small,
we can apply this decomposition to 
$\psi+\eta$, where $\eta$ is an element of $B_{\varepsilon}(\Lambda^4_{35};0)$, the ball of radius $\varepsilon$ centred at $0$ in $\Lambda^4_{35}$.
We can write this decomposition as
\begin{align}
\label{def-dec-tang-norm}
\psi_0 + \eta = \Pi(\eta) + \Theta(\eta),
\end{align}
with unique smooth maps
\begin{align}
\label{Pi-Theta}
\Pi: B_{\varepsilon}(\Lambda^4_{35};0) \rightarrow (\mathrm{GL}(8,\mathbb{R})\cdot\psi_0),
\quad
\Theta: B_{\varepsilon}(\Lambda^4_{35};0) \rightarrow \Lambda^4_{27}.
\end{align}
By the uniqueness we have $\Pi(0)=\psi_0$ and $\Theta(0)=0$.
Differentiating the path in $(\mathrm{GL}(8,\mathbb{R})\cdot\psi_0)$ given by $\Pi(t\eta)$ gives
\begin{align*}
\left.
\frac{d}{dt}
\right|_{t=0}
\Pi(t \eta)
=
\eta -
\left.
\frac{d}{dt}
\right|_{t=0}
\Theta(t \eta) 
\end{align*}
Taking the type decomposition with respect to $\psi_0$, with \eqref{tangent-space-orbit}
we see that the derivatives of $\Pi$ and $\Theta$ at $0$ are given by
\begin{align*}
D\Pi|_0 = \Id, \quad 
D\Theta|_0 = 0.
\end{align*}
$A \in \mathrm{Spin}(7)$ 
preserves the size and type of $\eta$ and thus
we have
\begin{align*}
A^*\Pi(\eta) + A^* \Theta(\eta) = \psi_0 + A^* \eta = \Pi(A^*\eta)+\Theta(A^*\eta). 
\end{align*}
By the uniqueness of the decomposition we see that the maps $\Pi$ and $\Theta$ are Spin(7)-equivariant.

To sum up, the maps \eqref{Pi-Theta} defined by the decomposition \eqref{def-dec-tang-norm} have the properties:
\begin{compactenum}[(i)]
\item $\Pi(0)=\psi_0$ and $\Theta(0)=0$,
\item $D\Pi|_0 = \Id$ and $D\Theta|_0 = 0$,
\item $\Pi$ and $\theta$ are Spin(7)-equivariant.
\end{compactenum}

Back to the global picture, on a Spin(7)-manifold $(M,\psi)$ we can piece
the fibre-wise maps together to define such maps in a $\varepsilon$-neighbourhood of the zero section in $\Lambda^4_{35}$. The fibre-wise norms are taken with respect to the inner product induced by $\psi$.

\subsection{Asymptotic types of non-compact Spin(7)-manifolds}

Let $\Sigma$ be a 7-manifold equipped with a complete Riemannian metric $g_{\ig \Sigma}$ which is induced by the $\mathrm{G}_2$-structure $\varphi_{\ig\Sigma}\in\Omega^3_+(\Sigma)$. Then the conical metric
\begin{align*}
g_{\ig C} = dr^2 + r^2 g_{\ig \Sigma}
\end{align*}
on $C(\Sigma)= (0,\infty) \times \Sigma$ is induced by the Spin(7)-structure
\begin{align*}
\psi_{\ig C}
=
r^3 dr\wedge\varphi_{\ig \Sigma} + r^4 *_{\ig\Sigma}\varphi_{\ig \Sigma}.
\end{align*}
$(C(\Sigma),\psi_{\ig C})$ is said to be a $\Spin(7)$-cone
if $\psi_{\ig C}$ is torsion-free. 
The exterior derivative is given by
\begin{align*}
d\psi_{\ig C}
=
-r^3 dr\wedge d\varphi_{\ig \Sigma}
+
4 r^3 dr\wedge *_{\ig \Sigma}\varphi_{\ig \Sigma} 
+
r^4 d*_{\ig\Sigma}\varphi_{\ig \Sigma}.
\end{align*}
Hence, the condition $d\psi_{\ig C}=0$ is equivalent
to
\begin{align}
d\varphi_{\ig \Sigma} = 4 *_{\ig \Sigma}\varphi_{\ig \Sigma}.
\label{NP-G2-eq}
\end{align}
This means that the $\mathrm{G}_2$-structure
on $\Sigma$ is \textit{nearly parallel}.
Nearly parallel $\mathrm{G}_2$-manifolds
are Einstein manifolds with positive scalar curvature. 
In particular, $\Sigma$ must be compact.
If the link is the 7-sphere with the round metric,
then the cone is the euclidean $\mathbb{R}^8$ with the standard Spin(7)-structure. Apart from its quotients, this is the only Spin(7)-cone with trivial holonomy. All other Spin(7)-cones need to have  holonomy group $\Sp(2)$, $\SU(4)$ or Spin(7). 
If the holonomy equals Sp(2), the link is a \textit{3-Sasakian}
manifold, and if it equals SU(4), then the link must be a 7-dimensional \textit{Sasaki-Einstein} manifold.
If the cone has full holonomy Spin(7), we say the nearly parallel 
$\mathrm{G}_2$-structure is \textit{proper}.

We are interested in Spin(7)-manifolds which are asymptotic to a Spin(7)-cone with a polynomial decay rate. By the Cheeger--Gromoll splitting theorem  irreducible non-compact Spin(7)-manifolds can have only one end. Therefore, we assume that $\Sigma$ is connected.

\begin{definition}
\label{def-AC-Spin(7)}
Let $C:=(C(\Sigma),\psi_{\ig C}, g_{\ig C})$ be the $\Spin(7)$-cone over the 
nearly parallel $\mathrm{G}_2$-manifold 
$(\Sigma,\varphi_{\ig\Sigma},g_{\ig\Sigma})$. A $\Spin(7)$-manifold $(M,\psi,g)$ is an \textit{asymptotically conical} (AC) $\Spin(7)$-manifold asymptotic to $C$
with rate $\nu\in(-\infty,0)$ if there exist a compact subset
$K\subset M$, $R > 0$ and a diffeomorphism 
\begin{align*}
F\colon (R,\infty)\times \Sigma \subset C(\Sigma) \rightarrow M-K
\end{align*}
such that we have the decay
\begin{align*}
|\nabla_{\ig C}^j(F^*\psi-\psi_{\ig C})|_{g_{\ig C}}=\mathcal{O}(r^{\nu-j})
\quad
\text{for all}\ j\in\mathbb{N}_0.
\end{align*}
In particular, this implies
\begin{align*}
|\nabla_{\ig C}^j(F^*g-g_{\ig C})|_{g_C}=\mathcal{O}(r^{\nu-j})
\quad
\text{for all}\ j\in\mathbb{N}_0.
\end{align*}
\end{definition}

\begin{definition}
\label{def-radial-function}
For a fixed $F$ as above we fix a radial function $\rho$ and a cut-off function $\chi$ for the remainder of this paper. 
\begin{itemize}
\item On the compact piece $K$ set $\rho \equiv 1$, on $F((R+1,\infty)\times \Sigma)$ set $\rho \equiv r$, and in the intermediate region interpolate smoothly in an increasing fashion.
In particular, $\rho \geq 1$ everywhere. 
\item $\chi\colon C(\Sigma)\rightarrow [0,1]$ is a cut-off function supported on 
$(R,\infty)\times \Sigma$ and $\chi\equiv 1$ on $(R+1,\infty)\times \Sigma$.
\end{itemize}
This will allow us to introduce weighted  function spaces.
Furthermore, if $\gamma$ is a differential form on $C(\Sigma)$
we can transplant it to $M$ via
$(F^{-1})^*(\chi \gamma)$.
By abuse of notation we will suppress $F$ in the rest of the paper and just write $\chi \gamma$ for the corresponding form on $M$.
\end{definition}

\begin{remark}
\label{Euclidean-space}
Suppose that the AC Spin(7)-manifold $(M,\psi,g)$ is asymptotic 
to the Euclidean Spin(7)-structure $\psi_0$ on $\mathbb{R}^8$.
Fix an arbitrary point $p\in M$. Denote the distance from $p$ by $r=\mathrm{dist}_g(\cdot, p)$ and the volume of a ball of radius $r$ in 8-dimensional Euclidean space by $v(r)$. The asymptotic behaviour of the metric implies that the function
\begin{align*}
r\mapsto \frac{\Vol\, B(p,r)}{v(r)}
\end{align*}
converges to 1 as $r\rightarrow \infty$.
However, by the Bishop--Gromov volume comparison theorem this function is non-increasing and converges to 1 as $r\rightarrow 0$. 
This shows that every ball of radius $r$ in $M$ has the same volume as a corresponding ball in Euclidean space. This implies that $(M,\psi,g)$
is isometric to $(\mathbb{R}^8,\psi_0,g_0)$. Therefore, we do not need to consider AC Spin(7)-manifolds asymptotic to Euclidean space, and exclude this case from all of our statements. 
\end{remark}

\subsection{Analysis on AC Spin(7)-manifolds}

Deformations of Spin(7)-structures have first been studied by Joyce on compact manifolds (see \cite[Section 10.7]{big-joyce}). It relies on analysis and Hodge theory on compact manifolds. To study deformations of other geometries it is essential to have an analytic framework adapted to the situation. 
E.g. Nordstr\"om studied deformations of $\mathrm{G}_2$- and Spin(7)-structures on EAC (exponentially asymptotically cylindrical) manifolds \cite{Nordstrom-PhD}.
We will need analysis on conifolds. In this section we collect the necessary analytical background. It is our aim to make this as concise as possible. For references for the statements in this section and a more detailed account of weighted analysis on conical and cylindrical spaces and its applications to geometry we refer the reader to \cite{KL,Nordstrom-PhD,Marshall,Pacini}
and
\cite[Appendix B]{FHN1}.
The underlying theory was outlined by Lockhart--McOwen \cite{Lockhart-McOwen}.

In the following $V$ and $W$ will be a subbundle of $\Lambda^{\bullet}T^*M$. Via the identification \eqref{spin-bundle-iso},
$\mathbf{S}_{\pm}(M)$, the positive and negative spinor bundle on $M$,
also fit into the discussion of this section.
The metric $g$ on $M$ and its Levi-Civita connection induce  a metric and a metric connection on $V$ and $W$. 
By $V_{\ig C}, W_{\ig C}$ we denote the corresponding vector bundles on the cone.

We set
\begin{align*}
\mathcal{C}^{\infty}_{\lambda}(V)
=
\{
\gamma \in \mathcal{C}^{\infty}(V)\, 
|
\,
| \nabla^j \gamma| = \mathcal{O}(r^{\lambda - j})
\,
\text{for all}\,
j \geq 0
\}.
\end{align*}

Next we define appropriate Banach spaces of sections of such bundles.
\begin{definition}
Let $p \geq 1$, $k\in\mathbb{N}_0$ and $\lambda\in\mathbb{R}$. For any $\gamma\in\mathcal{C}_{0}^{\infty}(V)$ the quantity
\begin{align*}
\| \gamma \|_{L^p_{k,\lambda}}
=
\left(
\sum_{j=0}^k
\int_M
|\rho^{-\lambda+j} \nabla^j \gamma|^p
\rho^{-8} \text{vol}_g
\right)^{\frac{1}{p}}
\end{align*}
is well defined and a norm. Here $\rho$ is the radial function from Definition \ref{def-radial-function}.
We define the weighted Sobolev space $L^p_{k,\lambda}(V)$ to be the completion of $\mathcal{C}^{\infty}_0(V)$ with respect to this norm.
$L^2_{k,\lambda}(V)$ is a Hilbert space with the inner product
\begin{align*}
\langle \gamma, \xi \rangle_{L^2_{k,\lambda}}
=
\sum_{j=0}^k
\int_M
 \langle \rho^{-\lambda+j} \nabla^j \gamma, \rho^{-\lambda+j} \nabla^j \xi \rangle
\rho^{-8} \text{vol}_g.
\end{align*}
\end{definition}
\begin{remark}
\label{relation-between-weighted-spaces}
\begin{compactenum}[(i)]
\item
Note that $L^2_{0,-4}(V)=L^2(V)$. We refer to rates $\nu < -4$ as the \textit{$L^2$-setting} and to rates $\nu > -4$ as the \textit{non-$L^2$ setting}.
\item 
From the definition it follows that
$\rho^{\mu} L^2_{0,\lambda}(V)=L^2_{0,\lambda+\mu}(V)$.
In particular, we have $L^2_{0,\lambda}(V)=\rho^{-4-\lambda}L^2$.
\item
Set $\Omega^k_{l,\lambda}:=L^2_{l,\lambda}(\Lambda^k)$,
$\Omega^{\mathrm{even}}_{l,\lambda}:=L^2_{l,\lambda}(\Lambda^{\mathrm{even}})$ and $\Omega^{\mathrm{odd}}_{l,\lambda}:=L^2_{l,\lambda}(\Lambda^{\mathrm{odd}})$.
\end{compactenum}
\end{remark}

\begin{definition}
Let $p \geq 1$, $k\in\mathbb{N}_0$ and $\lambda\in\mathbb{R}$. For any $\gamma\in\mathcal{C}_{0}^{\infty}(V)$ the quantity
\begin{align*}
\| \gamma \|_{\mathcal{C}^{k,\alpha}_{\lambda}}
=
\sum_{j=0}^k
\| \rho^{-\lambda+j} \nabla^j \gamma \|_{\mathcal{C}^0}
+
[
\rho^{-\lambda+k}\nabla^k \gamma
]_{\alpha}
\end{align*}
is well defined and a norm. 
Here $[\ \cdot \ ]_{\alpha}$ is the H\"older seminorm.
We define the weighted H\"older space $\mathcal{C}^{k,\alpha}_{\lambda}(V)$ to be the closure of $\mathcal{C}^{\infty}_0(V)$ with respect to this norm.
\end{definition}

\begin{theorem}
\cite[Theorem 4.17]{Marshall}
\label{Sobolev embedding}
\begin{compactenum}[(i)]
\item
If $l\geq m+\alpha+4$, then there is a continuous embedding 
$L^2_{l,\lambda}(V)\hookrightarrow \mathcal{C}^{m,\alpha}_{\lambda}(V)$.
\item
If $\lambda < \lambda'$ and $l > 0$, there is a compact embedding 
$L^2_{l,\lambda}\hookrightarrow L^2_{0,\lambda'}$.
\end{compactenum}
\end{theorem}

In the following denote by $*_{\ig M}$ the Hodge star operator on $M$ induced by $g$, and by $*_{\ig C}$ the Hodge star operator on the asymptotic cone $C(\Sigma)$ induced by the conical metric $g_{\ig C}$. $d^*_{\ig M}$ and $d^*_{\ig C}$ denote the co-differential on $M$ and $C(\Sigma)$, respectively. In our calculations it is useful to know that comparing the Hodge star operator on $M$ and the asymptotic cone gives an additional decay of $\nu$, the AC rate.

\begin{lemma}
\label{additional-decay}
Suppose $\gamma\in\Omega^k_{l,\lambda}$. Then
\begin{itemize}
\item $(*_{\ig M}-*_{\ig C})\gamma\in\Omega^{8-k}_{l,\lambda+\nu}$.
\item $(d+d^*_{\ig M})\gamma -(d+d^*_{\ig C})\gamma\in\Omega^{k-1}_{l-1,\lambda+\nu-1}\oplus \Omega^{k+1}_{l-1,\lambda+\nu-1}$.
\end{itemize}
Suppose $\gamma\in\mathcal{C}^{\infty}_{\lambda}(\Lambda^k)$. Then
\begin{itemize}
\item $(*_{\ig M}-*_{\ig C})\gamma\in\mathcal{C}^{\infty}_{\lambda+\nu}(\Lambda^{8-k})$.
\item $(d+d^*_{\ig M})\gamma -(d+d^*_{\ig C})\gamma\in\mathcal{C}^{\infty}_{\lambda+\nu-1}(\Lambda^{k-1}\oplus \Lambda^{k+1})$.
\end{itemize}
\end{lemma}

\begin{proposition}
\label{L2-pairing}
Suppose $\eta\in L^2_{0,\lambda}(V)$ and $\omega\in L^2_{0,\mu}(V)$. If $\lambda + \mu \leq -8$, then the $L^2$-pairing 
\begin{align*}
\langle \eta, \omega \rangle_{L^2}
=
\int_M \langle \eta, \omega \rangle \Vol
\end{align*}
is finite and satisfies the inequality
\begin{align*}
\langle \eta, \omega \rangle_{L^2}
\leq 
\| \eta \|_{L^2_{0,\lambda}} \| \omega \|_{L^2_{0,\mu}}
\end{align*}
\end{proposition}
\begin{proof}
Using the Cauchy-Schwarz inequality both for the pointwise inner product and the $L^2$-version, $\lambda + \mu \leq -8$ and $\rho \geq 1$,
we get
\begin{align*}
\langle \eta, \omega \rangle_{L^2}
&=
\int_M \langle \eta, \omega \rangle \Vol
\\
&\leq
\int_M |\eta| |\omega| \Vol
\\
&=
\int_M
(|\rho^{-\lambda} \eta| \rho^{-4})(|\rho^{-\mu} \omega | \rho^{-4})
\rho^{8+\lambda+\eta} \Vol
\\
&\leq 
\int_M
(|\rho^{-\lambda} \eta| \rho^{-4})(|\rho^{-\mu} \omega | \rho^{-4})
\Vol 
\\
&\leq 
\left(
\int_M |\rho^{-\lambda} \eta|^2 \rho^{-8} \Vol
\right)^{1/2}
\left(
\int_M |\rho^{-\mu} \omega|^2 \rho^{-8} \Vol
\right)^{1/2}
\\
&=
\| \eta \|_{L^2_{0,\lambda}} \| \omega \|_{L^2_{0,\mu}}
< \infty.
\end{align*}
\end{proof}

\begin{proposition}
\label{weighted-L2-dual}
We have $(L^2_{0,\lambda}(V))^* \cong L^2_{0,-8-\lambda}(V)$. 
\end{proposition}
\begin{proof}
Proposition \ref{L2-pairing} gives a pairing 
\begin{align*}
\langle \cdot, \cdot \rangle_{L^2}\colon
L^2_{0,\lambda}(V) \times L^2_{0,-8-\lambda}(V)
\rightarrow \mathbb{R}.
\end{align*}
This defines a continuous linear map 
\begin{align*}
L^2_{0,-8-\lambda}(V) \rightarrow (L^2_{0,\lambda}(V))^*,
\quad
\omega \mapsto \langle \cdot, \omega \rangle_{L^2}.
\end{align*}
Under the Hilbert space isomorphism
$(L^2_{0,\lambda}(V))^* \cong L^2_{0,\lambda}(V)$ this corresponds to the map
\begin{align*}
L^2_{0,-8-\lambda}(V) \rightarrow L^2_{0,\lambda}(V),
\quad
\omega\mapsto \rho^{8+2\lambda} \omega.
\end{align*}
This clearly is an isomorphism.
\end{proof}

In the remainder of the section let
\begin{align*}
P: \Gamma(V) \rightarrow \Gamma(W)
\end{align*}
be one of the elliptic differential operators 
\begin{align}
\label{d+d*}
d+d^* &: \Gamma(\Lambda^{\bullet})\rightarrow \Gamma(\Lambda^{\bullet}),
\\
\label{Laplace}
\Delta &: \Gamma(\Lambda^k) \rightarrow \Gamma(\Lambda^k),
\\
\slashed{D}_+ &: \Gamma(\mathbf{S}_+)\rightarrow \Gamma(\mathbf{S}_-),
\\
\label{D-}
\slashed{D}_- &: \Gamma(\mathbf{S}_-)\rightarrow \Gamma(\mathbf{S}_+),
\end{align}
Denote by $k$ the order of $P$ and by $P_{\ig C}$ the corresponding differential operator 
\begin{align*}
P_{\ig C}: \Gamma(V_{\ig C}) \rightarrow \Gamma(W_{\ig C})
\end{align*}
on the cone.  
$P$ is asymptotic to $P_{\ig C}$ in an appropriate sense and thus is an example of an \textit{asymptotically conical operator} \cite[section 4.3.2]{Marshall}. 
In this case $P$ extends to a bounded linear map on all weighted Sobolev and H\"older spaces \cite[Proposition 4.20]{Marshall}.
By $P_{l+k,\lambda}$ we denote the induced operator 
\begin{align}
P_{l+k,\lambda}:
L^2_{l+k,\lambda}(V)\rightarrow
L^2_{l,\lambda-k}(W)
\label{operator on sobolev spaces}
\end{align}
on weighted Sobolev spaces. 

Denote by $P^*$ the formal adjoint of $P$. 
\eqref{d+d*} and \eqref{Laplace} are formally self-adjoint and $\slashed{D}_+^*=\slashed{D}_-$.
The next Lemma shows that integration by parts is available for sections of weighted Sobolev spaces if the decay rates are fast enough.

\begin{lemma}\label{partial integration}
Let 
$\eta\in L^2_{k,\lambda}(V)$ and $\omega\in L^2_{k,\mu}(V)$.
If $\lambda+\mu \leq -8+k$, then  
the quantities $\langle
P\eta, \omega 
\rangle_{L^2}$ 
and $ \langle
\eta, P^* \omega
\rangle_{L^2}$ are finite and equal, i.e.
\begin{align*}
\langle
P\eta, \omega 
\rangle_{L^2}
=
\langle
\eta, P^* \omega
\rangle_{L^2}.
\end{align*}
\end{lemma}
\begin{proof}
We first show that the two quantities are finite.
If $\lambda+\mu\leq -8+k$, by Proposition \ref{L2-pairing} and the continuity of $P_{k,\lambda}$ we have
\begin{align*}
\langle P\eta, \omega \rangle_{L^2}
\leq 
\| P\eta \|_{L^2_{0,\lambda-k}}
\| \omega \|_{L^2_{0,\mu}}
\leq C
\| \eta \|_{L^2_{k,\lambda}}
\| \omega \|_{L^2_{k,\mu}}
< \infty.
\end{align*}
Finiteness of $\langle \eta, P^* \omega \rangle$ follows analogously.
Hence by the dominated convergence theorem it is enough to prove the statement for compactly supported forms, for which it is true by the definition of the formal adjoint.
\end{proof}
\begin{remark}
Lemma \ref{partial integration} says that if we consider 
the operators 
\begin{align*}
P_{k,\lambda}: L^2_{k,\lambda}(V) &\longrightarrow L^2_{0,\lambda-k}(W),
\\
L^2_{0,-8-\lambda}(V) &\longleftarrow L^2_{k,-8-\lambda+k}(W) :P^*_{k,-8-\lambda+k},
\end{align*}
then with respect to the pairings 
\begin{align*}
L^2_{0,\lambda}(V)\times L^2_{0,-8-\lambda}(V)\rightarrow \mathbb{R},
\quad
L^2_{0,\lambda-k}(W) \times L^2_{0,-8-\lambda+k}(W)\rightarrow \mathbb{R},
\end{align*}
the operator $P^*_{k,-8-\lambda+k}$ is the ``adjoint'' of $P_{k,\lambda}$. Using the Fourier transform, one can define Sobolev spaces $L^2_{s,\mu}$ for any $s\in\mathbb{R}$. 
Then similarly to Proposition \ref{weighted-L2-dual} we have
$(L^2_{k,\lambda}(V))^*=L^2_{-k,-8-\lambda}(V)$
and $(L^2_{0,\lambda-k}(W))^*=L^2_{0,-8-\lambda+k}(W)$.
Therefore, $P^*_{k,-8-\lambda+k}$ is the restriction of the full adjoint $(P_{k,\lambda})^*=P^*_{0,-8-\lambda+k}$ to the more regular subspace $L^2_{k,-8-\lambda+k}(W)$. In all cases in this paper sections in the kernel and cokernel are smooth by elliptic regularity. Therefore, we do not have to  use the full adjoint.
\end{remark}

In the examples \eqref{d+d*}-\eqref{D-} that we consider, both $P$ and the asymptotic model $P_{\ig C}$ are elliptic. 
Thus $P$ is an example of an  \textit{uniformly elliptic} operator \cite[Section 4.3.2]{Marshall}. This control at infinity allows to develop a Fredholm and regularity theory similar to the setting of compact manifolds.

\begin{theorem}
\cite[Theorem 4.21]{Marshall}
\label{ellitptic regularity}
Suppose  $\gamma\in L^1_{\mathrm{loc}}(V)$ is
a weak solution of $P\gamma = \zeta$ with $\zeta\in L^1_{\mathrm{loc}}(W)$. If $\zeta$ lies in $L^2_{l,\lambda-k}(W)$ $($respectively $\mathcal{C}^{l,\alpha}_{\lambda-k}(W))$, 
then $\gamma$ lies in $L^2_{l+k,\lambda}(V)$ $($respectively  $\mathcal{C}^{l+k,\alpha}_{\lambda}(V))$, and is a strong solution. Furthermore, there exists a positive constant $C>0$ such that $\gamma$ satisfies the estimate
\begin{align}
\label{elliptic-estimate-sobolev}
\| \gamma \|_{L^2_{l+k,\lambda}}
\leq 
C
\left(
\|P\gamma\|_{L^2_{l,\lambda-k}}+\| \gamma \|_{L^2_{0,\lambda}}
\right),
\end{align}
respectively the estimate
\begin{align}
\| \gamma \|_{\mathcal{C}^{l+k,\alpha}_{\lambda}}
\leq 
C
\left(
\|P\gamma\|_{\mathcal{C}^{l,\alpha}_{\lambda-k}}+\| \gamma \|_{\mathcal{C}^{0,\alpha}_{\lambda}}
\right).
\end{align}
\end{theorem}

\begin{remark}
Using Theorems \ref{ellitptic regularity} and \ref{Sobolev embedding} 
it follows immediately that any closed and coclosed form on $(M,\psi)$ is smooth. Therefore we can denote the kernel
of 
\begin{align*}
P_{l+k,\lambda}:
L^2_{l+k,\lambda}(V)\rightarrow
L^2_{l,\lambda-k}(W)
\end{align*}
by $\mathrm{ker}P_{\lambda}$.
\end{remark}

We will also have to deal with less regular linear differential operators $L: \Gamma(V)\rightarrow \Gamma(W)$ for which the coefficients only lie in some H\"older space $\mathcal{C}^{l,\alpha}$,
and the convergence to the asymptotic model is with respect to the $\mathcal{C}^{l,\alpha}_{\lambda}$-norm. The above elliptic regularity result generalises to this setting. 

\begin{theorem}
\cite[Theorem 4.2.22]{Nordstrom-PhD}
\label{general-ell-reg}
Let $L: \Gamma(V) \rightarrow \Gamma(W)$ be a linear elliptic differential operator of rank $k$ with $\mathcal{C}^{l,\alpha}$-regular coefficients which is $\mathcal{C}^{l,\alpha}_{\lambda}$-asymptotic to the conical, elliptic differential operator $L_{\ig C}$. If $u\in \mathcal{C}^{k,\alpha}_{\lambda+k}$ and $Lu\in\mathcal{C}^{l,\alpha}_{\lambda}$, then $u\in\mathcal{C}^{k+l,\alpha}_{\lambda+k}$ and
\begin{align*}
\| u \|_{\mathcal{C}^{l+k,\alpha}_{\lambda+k}}
\leq 
C
\left(
\|Lu\|_{\mathcal{C}^{l,\alpha}_{\lambda}}+\| u \|_{\mathcal{C}^{0,\alpha}_{\lambda+k}}
\right).
\end{align*} 
\end{theorem}

The key in understanding the mapping properties of $P_{l+k,\lambda}$ is to study $P_{\ig C}$ acting on homogeneous sections.

\begin{definition}
$\gamma\in \Omega^k(C(\Sigma))$ is \textit{homogeneous} of rate $\lambda$ if there exist $\alpha\in\Omega^{k-1}(C(\Sigma))$ and $\beta\in\Omega^k(C(\Sigma))$ such that
\begin{align*}
\gamma
=
r^{\lambda}(r^{k-1} dr\wedge\alpha + r^k \beta)
.
\end{align*} 
$\gamma\in\Omega^{\bullet}(C(\Sigma))$ is
said to be homogeneous of rate $\lambda$ if each degree component is homogeneous of rate $\lambda$. 
In both cases this implies that $|\gamma|_{g_C}$ is a homogeneous function in $r$ of rate $\lambda$.
\end{definition}

\begin{definition}
\label{def-K}
$\lambda\in \mathbb{R}$ is a \textit{critical rate} for $P$ if there exists a non-zero homogeneous section $\gamma$ of rate $\lambda$ such that $P_{\ig C}\gamma = 0$. Denote by 
\begin{align*}
\mathcal{D}(P)
=
\{
\lambda\in\mathbb{R}\, |\, 
\exists \gamma\in\mathcal{C}^{\infty}(V_{\ig C})\, \text{non-zero and homogeneous of rate}\ \lambda\,\text{such that}\, P_{\ig C}\gamma = 0
\}
\end{align*} 
the set of all critical rates. 
Set
\begin{align*}
\mathcal{K}_{P}(\lambda)
=
\{
\gamma = \sum_{j=0}^m (\log r)^j \gamma_j\,|\,
\text{each}\ \gamma_j\ \text{is homogeneous of order}\
\lambda\ \text{and}\ P_{\ig C}\gamma = 0
\}.
\end{align*}
$\mathcal{K}_{\Lambda^k}(\lambda)$,
$\mathcal{K}_{\mathrm{even}}(\lambda)$, $\mathcal{K}_{\mathrm{odd}}(\lambda)$ and $\mathcal{K}_{\mathrm{ASD}}(\lambda)$ are defined by choosing $P$ to be $d+d^*$ acting on k-forms, even degree forms, odd degree forms and anti-self dual 4-forms, respectively.
\end{definition}
\begin{remark}
\begin{compactenum}[(i)]
\item 
In the general theory one also has to consider complex critical rates. However, all operator considered by us are formally self-adjoint, or restrictions thereof. In this case all critical rates need to be real.
\item
An important property of the set 
$\mathcal{D}(P)$ is that it is discrete.
\item
Suppose $ \gamma=\sum_{j=0}^m (\log r)^j \gamma_j$, where each $\gamma_j$ is homogeneous of order $\lambda$ and $\gamma_m \neq 0$. As a polynomial in $\log r$, the leading order term of $P_{\ig C} \gamma$ is $P_{\ig C} \gamma_m$. Hence $\mathcal{K}_{P}(\lambda)=\{0\}$ if $\lambda$ is not a critical rate.
\end{compactenum}
\end{remark}

Contrary to the compact setting, the elliptic estimate \eqref{elliptic-estimate-sobolev} cannot be used to prove that $P_{l+k,\lambda}$ is a Fredholm operator because $L^2_{l+k,\lambda}$ does not embed compactly into $L^2_{0,\lambda}$.
However, if $\lambda$ is a non-critical rate, the estimate can be strengthened to

\begin{proposition}
\cite[Proposition B.9]{FHN1}
\label{prop-improved-estimate}
If $\lambda' > \lambda$ and the interval $[\lambda,\lambda']$ does not contain a critical rate for $P$, then for any $\gamma\in L^2_{l+r,\lambda}(V)$ we have
\begin{align}
\label{improved-elliptic-estimate}
\| \gamma \|_{L^2_{l+k,\lambda}}
\leq 
C
\left(
\|P\gamma\|_{L^2_{l,\lambda-k}}+\| \gamma \|_{L^2_{0,\lambda'}}
\right).
\end{align}
\end{proposition}

The same statement holds for H\"older spaces, but we mainly use Sobolev spaces.
In the improved estimate \eqref{improved-elliptic-estimate} the ``error term'' $\| \gamma \|_{L^2_{0,\lambda'}}$ is the norm of a space into which $L^2_{l+k,\lambda}$ embeds compactly by Theorem \ref{Sobolev embedding} (ii). This allows us to prove

\begin{corollary}
\label{poincare inequality}
Assume that $\lambda$ is a non-critical rate for $P$. Then there exists a positive constant $C$ such that for all
$\gamma\in L^2_{l+k,\lambda}(V)$ which are orthogonal in $L^2_{l+k,\lambda}(V)$ to $\ker(P_{l+k,\lambda})$ we have
\begin{align*}
\| \gamma \|_{L^2_{l+k,\lambda}}
\leq 
C
\|P\gamma\|_{L^2_{l,\lambda-k}}
.
\end{align*}
\end{corollary} 

From Theorem \ref{Sobolev embedding} (ii), Proposition \ref{prop-improved-estimate} and Corollary \ref{poincare inequality} we finally can clarify the mapping properties of $P_{l+k,\lambda}$.

\begin{theorem}
\cite[Theorems 4.11 and 4.13]{KL}
\label{Fredholm}
The operator
\begin{align*}
P_{l+k,\lambda}:
L^2_{l+k,\lambda}(V)\rightarrow
L^2_{l,\lambda-k}(W)
\end{align*}
is Fredholm if and only if $\lambda$ is not a critical rate. 
In this case we have:
\begin{compactenum}[(i)]
\item
We can identify a complement of $P_{l+k,\lambda}(L^2_{l+k,\lambda}(V))$ in $L^2_{l,\lambda-k}(W)$ with $\ker P^*_{-8-\lambda+k}$, i.e. there exists a finite dimensional subspace $U$
of $L^2_{l,\lambda-k}(W)$
such that 
\begin{align*}
L^2_{l,\lambda-k}(W)
=
P_{l+k,\lambda}(L^2_{l+k,\lambda}(V)) \oplus U
\end{align*}
and $U \cong \ker P^*_{-8-\lambda+k}$.
\item
If $\lambda \geq -4+k$, then $U$ is a subspace of $L^2_{l,\lambda-k}(W)$ and we can set $U = \ker P^*_{-8-\lambda+k}$.
\end{compactenum}
\end{theorem}

\noindent
For us it will be important to understand how $\mathrm{ker}P_{\lambda}$ changes as $\lambda$ varies.
The key to understanding how the kernel changes at a critical rate is the following

\begin{theorem}
\cite[Proposition 4.21]{KL}
\label{Lockhart-McOwen-key}
Let $\lambda_2 < \lambda_1$ be two non-critical rates of $P$
and $\lambda_0$ the only critical rate in the interval $(\lambda_2,\lambda_1)$. If $\gamma \in L^2_{l,\lambda_1}(V)$
and $P\gamma \in L^2_{l-k,\lambda_2-k}(W)$  (i.e. ``$P\gamma$ decays  faster than expected''), then there exist 
$\eta\in \mathcal{K}(\lambda_0)_{P_C}$ and $\tilde{\gamma}\in L^2_{l,\lambda_0+\nu}(V)$ which depend both linearly on $\gamma$, such that outside of a compact subset
\begin{align*}
\gamma - \eta - \tilde{\gamma} \in L^2_{l,\lambda_2}(V).
\end{align*}
\end{theorem}

\noindent
A consequence of Theorem \ref{Lockhart-McOwen-key} we obtain
the following two theorems.

\begin{theorem}
\cite[Theorem 4.20]{KL}
\label{index change for elliptic operators}
For non-critical rates $\lambda_2 < \lambda_1$ of $P$
the index change is given by
\begin{align*}
\ind P_{\lambda_1} - \ind P_{\lambda_2}
=
\sum_{\lambda\in (\lambda_2,\lambda_1) \cap \mathcal{D}(P)}
\dim \mathcal{K}_{P}(\lambda).
\end{align*}
\end{theorem}

\begin{theorem}
\label{constant away from crit rates}
If the interval $[\lambda_1,\lambda_2]$ does not contain any critical rate, then $\mathrm{ker}P_{\lambda_1} = \mathrm{ker}P_{\lambda_2}$.
\end{theorem}

\subsection{Differential forms on AC Spin(7)-manifolds}

\subsubsection{Computations on the asymptotic Spin(7)-cone}

In this paper it is important to have a good understanding of the Laplace-operator and $d+d^*$ acting on forms in weighted Sobolev spaces. In light of Theorem \ref{Fredholm}
we need to understand the critical rates of the corresponding operators acting on homogeneous sections on the cone.
We first collect some explicit formulas. 
Let $ \gamma = r^{\lambda} (r^{k-1}dr\wedge \alpha + r^k \beta)$ be a homogeneous $k$-form on $C(\Sigma)$,
where $\alpha\in \Omega^{k-1}(\Sigma)$ and $\beta\in \Omega^k(\Sigma)$.
The Hodge-star operator on the cone is given by
\begin{gather*}
\ast_{\ig C}(dr\wedge \alpha)
=
r^{9-2k}\ast_{\ig\Sigma}\alpha,
\\
\ast_{\ig C}\beta 
= (-1)^k r^{7-2k} dr\wedge \ast_{\ig\Sigma}\beta.
\end{gather*}
Using this we get
\begin{subequations}
\label{formulas-cone}
\begin{align}
d\gamma
=&\,
r^{\lambda-1}
\Big(
r^k dr\wedge
(
(\lambda+k)\beta-d_{\ig \Sigma}\alpha
)
+
r^{k+1}
d_{\ig \Sigma}\beta
\Big),
\\[10pt]
\ast_{\ig C} \gamma
=&\,
r^{\lambda+8-k}\ast_{\ig\Sigma}\alpha 
+
(-1)^k r^{\lambda+7-k} dr\wedge \ast_{\ig\Sigma}\beta,
\\[10pt]
d_{\ig C}^*\gamma 
=&\,
r^{\lambda+k-2}
(
-(\lambda+8-k)\alpha + d^*_{\ig\Sigma}\beta
)
+
r^{\lambda+k-3} dr\wedge (-d^*_{\ig\Sigma}\alpha),
\\[10pt]
\Delta_{\ig C} \gamma
=&\,
r^{\lambda+k-3}
dr
\wedge
\Big(
\Delta_{\ig \Sigma}\alpha
-
(\lambda+k-2)(\lambda-k+8)\alpha
-
2d^*_{\ig \Sigma}\beta
\Big)
\\
\nonumber
& +
r^{\lambda+k-2}
\Big(
\Delta_{\ig \Sigma}\beta
-
(\lambda+k)(\lambda-k+6)\beta
-
2d_{\ig \Sigma}\alpha
\Big)
\end{align}
\end{subequations}

All of these formulas purely depend on the dimension of the cone. We make no use of the fact that our cone is a Spin(7)-cone.

\begin{remark}
\label{remark-char-asd}
For later use 
we give a brief characterisation of closed, homogeneous anti-self-dual 4-forms on the cone. Suppose 
\begin{align*}
\gamma = r^{\lambda+3} dr\wedge \alpha + r^{\lambda+4} \beta
\end{align*}
is a homogeneous 4-form of rate $\lambda$ for $\alpha\in\Omega^3(\Sigma)$ and $\beta\in\Omega^4(\Sigma)$. Then $\gamma$ is anti-self-dual if and only if $\beta = - *_{\ig\Sigma}\alpha$, i.e.
\begin{align*}
\gamma = r^{\lambda+3} dr\wedge \alpha + r^{\lambda+4} (-*_{\ig\Sigma}\alpha).
\end{align*}
If $\lambda\neq -4$, $\gamma$ is closed if and only if 
$d_{\ig\Sigma}\alpha = -(\lambda+4) *_{\ig \Sigma}\alpha$. Then we have in particular
\begin{align*}
\gamma
=
d\left(\frac{1}{\lambda+4}r^{\lambda+4}\alpha\right).
\end{align*}
If $\lambda=-4$, then $\gamma$ is closed if and only if $\alpha$ is harmonic on $\Sigma$.
\end{remark}

Later on we need to know
$\mathcal{K}_{\mathrm{even}}(-4)$ and $\mathcal{K}_{\mathrm{odd}}(-3)$.
The calculation is analogous to \cite[Proposition 2.21]{Karigiannis-desing} in the $\mathrm{G}_2$-setting.

\begin{lemma}\label{even forms at -4}
Let $\eta = \sum_{k=0}^4 \eta_{2k}$ be 
a closed and coclosed even degree form on $C(\Sigma)$ homogeneous of rate $-4$, i.e. 
\begin{align*}
\eta_{2k}=
r^{2k-5} dr\wedge \alpha_{2k-1}
+
r^{2k-4} \beta_{2k},
\end{align*}
where $\alpha_{2k-1}\in\Omega^{2k-1}(\Sigma)$
and $\beta_{2k}\in\Omega^{2k}(\Sigma)$.
Then all components except $\alpha_3$ and $\beta_4$ vanish, i.e. 
\begin{align*}
\eta = r^{-1} dr\wedge \alpha_3 + \beta_4,
\end{align*}
and $\alpha_3$ and $\beta_4$ are both harmonic on $\Sigma$.
In particular, $\eta$ is of pure degree 4.
\end{lemma}
\begin{proof}
We have 
\begin{align*}
d^*\eta_{2k}
=
r^{2k-7} dr\wedge(-d^*\alpha_{2k-1})
+
r^{2k-6} 
(-(4-2k)\alpha_{2k-1}+d^*\beta_{2k})
\end{align*}
and 
\begin{align*}
d\eta_{2k-2}
=
r^{2k-7} dr\wedge (-d\alpha_{2k-3} + (2k-6) \beta_{2k-2})
+
r^{2k-6} d\beta_{2k-2}.
\end{align*}
Because of $d^*\eta_{2k}+d\eta_{2k-2}=0$ we get
\begin{align*}
-d\alpha_{2k-3}+(2k-6)\beta_{2k-2}-d^*\alpha_{2k-1} = 0,
\\
d\beta_{2k-2}-(4-2k)\alpha_{2k-1} + d^*\beta_{2k}=0.
\end{align*}
This gives
\begin{align}
(k-4)\beta_k = d\alpha_{k-1}+d^*\alpha_{k+1},
\quad k=0,2,4,6,8,
\label{-4even1}
\\
(3-k)\alpha_k
=
d\beta_{k-1}+d^*\beta_{k+1},
\quad k=1,3,5,7.
\label{-4even2}
\end{align}
Applying $d$, $d^*$ to
(\ref{-4even1}) and (\ref{-4even2}) gives
\begin{align}
\Delta\alpha_k
=
(k-5)d\beta_{k-1} +(k-3) d^*\beta_{k+1},
\quad k=1,3,5,7,
\label{-4even3}
\\
\Delta\beta_k
=
(4-k)d\alpha_{k-1} + (2-k)d^*\alpha_{k+1},
\quad k=0,2,4,6,8.
\label{-4even4}
\end{align}
Combining (\ref{-4even1}),
(\ref{-4even2}), (\ref{-4even3})
and (\ref{-4even4}) gives
\begin{align}
\Delta\alpha_k
=
-(k-3)^2\alpha_k -2 d\beta_{k-1},
\quad k=1,3,5,7,
\label{-4even5}
\\
\Delta\beta_k
=
(2-k)(k-4)\beta_k + 2 d\alpha_{k-1},
\quad 
k=0,2,4,6,8.
\label{-4even6}
\end{align}
Now equations (\ref{-4even5}), (\ref{-4even6})
give successively:
$\Delta \beta_0 = -8 \beta_0$ and hence $\beta_0=0$.
$\Delta\alpha_1=-4\alpha_1$ and therefore $\alpha_1=0$.
Then $\beta_2, \alpha_3,\beta_4$ are harmonic.
For $k=5$ we get $\Delta\alpha_5 = -4\alpha_5$
and $\alpha_5$ vanishes. Then $\beta_6=0$ because $\Delta\beta_6 = -8 \beta_6$. $\Delta\alpha_7=-16\alpha_7$ gives $\alpha_7 = 0$
and finally $\Delta\beta_8 = -24 \beta_8$ gives $\beta_8 = 0$.
It is left to prove that $\beta_2$ vanishes. This immediately 
follows from (\ref{-4even1}).
\end{proof}

\begin{lemma}\label{odd-degree at rate -3}
Let $\eta = \sum_{k=0}^3 \eta_{2k+1}$ be a closed and coclosed odd degree form on $C(\Sigma)$ homogeneous of rate $-3$,
i.e. 
\begin{align*}
\eta_{2k+1}
=
r^{2k-3} dr\wedge \alpha_{2k} + r^{2k-2} \beta_{2k+1},
\end{align*}
where $\alpha_{2k}\in\Omega^{2k}(\Sigma)$ and 
$\beta_{2k+1}\in\Omega^{2k+1}(\Sigma)$.
Then all components except $\alpha_4$ and $\beta_3$ vanish, i.e. 
\begin{align*}
\eta = 
\eta_3 + \eta_5
=
\beta_3 + r dr\wedge \alpha_4,
\end{align*}
and $\alpha_4$ and $\beta_3$ are both harmonic on $\Sigma$.
In particular, $\eta_3$ and $\eta_5$ are individually closed and co-closed.
\end{lemma}
\begin{proof}
We have 
\begin{align*}
d^*\eta_{2k+1}
=
r^{2k-5} dr\wedge(-d^*\alpha_{2k})
+
r^{2k-4} (-(5-(2k+1))\alpha_{2k} + d^*\beta_{2k+1})
\end{align*}
and
\begin{align*}
d\eta_{2k-1}
=
r^{2k-5} dr\wedge(-d\alpha_{2k-2}+(2k-4)\beta_{2k-1})
+
r^{2k-4} d\beta_{2k-1}.
\end{align*}
Because of $d^*\eta_{2k+1}+d\eta_{2k-1} = 0$ we get
\begin{align*}
-d\alpha_{2k-2}+(2k-4)\beta_{2k-1}-d^*\alpha_{2k} = 0,
\\
-(5-(2k+1))\alpha_{2k} + d^*\beta_{2k+1}+d\beta_{2k-1} = 0.
\end{align*}
This gives 
\begin{align}
(k-3)\beta_k = d\alpha_{k-1} + d^*\alpha_{k+1}, 
\quad k=1,3,5,7,
\label{-3odd1}
\\
(4-k)\alpha_k
=
d\beta_{k-1}+d^*\beta_{k+1},\quad k=0,2,4,6.
\label{-3odd2}
\end{align}
Applying $d$ and $d^*$ to (\ref{-3odd1}) and (\ref{-3odd2}) gives
\begin{align}
\Delta\alpha_k
=
(k-4)d\beta_{k-1} +(k-2) d^*\beta_{k+1},
\quad k=0,2,4,6,
\label{-3odd3}
\\
\Delta\beta_k
=
(5-k)d\alpha_{k-1} + (3-k)d^*\alpha_{k+1},
\quad k=1,3,5,7.
\label{-3odd4}
\end{align}
Combining
(\ref{-3odd1}), (\ref{-3odd2}), (\ref{-3odd3}) 
and (\ref{-3odd4}) gives
\begin{align}
\Delta\alpha_k
=
-(k-2)(k-4)\alpha_k -2 d\beta_{k-1},
\quad k=0,2,4,6,
\label{-3odd5}
\\
\Delta\beta_k
=
-(k-3)^2\beta_k + 2 d\alpha_{k-1},
\quad 
k=1,3,5,7.
\label{-3odd6}
\end{align}
For $k=0$ we get $\Delta\alpha_0 = -8 \alpha_0$
and therefore $\alpha_0=0$.
For $k=1$
we get
$
\Delta\beta_1
=
-4\beta_1 $
and consequentially $\beta_1$ vanishes.
By (\ref{-3odd5})
we first see that $\alpha_2$ is harmonic.
This in turn by (\ref{-3odd6}) means that $\beta_3$ is harmonic.
Again by (\ref{-3odd5}) $\alpha_4$ is harmonic.
$k=5$ gives $\Delta\beta_5 = -4 \beta_5$ and hence $\beta_5=0$.
Then $\Delta\alpha_6=-8\alpha_6$ which gives $\alpha_6=0$.
Finally $\beta_7$ vanishes because $\Delta\beta_7 = -16\beta_7$.
We still have to show that $\alpha_2=0$. This now follows from
(\ref{-3odd2}).
\end{proof}

\begin{lemma}{\cite[Proposition A.7]{FHN1}}
\label{no-log}
Let $\gamma = \sum_{j=0}^m (\log r)^j \gamma $, where all $\gamma_j$ are differential forms on $C$ homogeneous of order $\lambda$. If $(d+d^*)\gamma =0$, then $m=0$.
\end{lemma}

As a consequence of Lemmas \ref{even forms at -4}, \ref{odd-degree at rate -3}
and \ref{no-log} we get

\begin{corollary}\label{forms on cone at rate -4}
We have
\begin{align*}
\mathcal{K}_{\mathrm{even}}(-4) &= r^{-1} dr\wedge H^3(\Sigma,\mathbb{R}) + H^4(\Sigma,\mathbb{R}),
\\
\mathcal{K}_{\mathrm{odd}}(-3) &=  H^3(\Sigma,\mathbb{R}) + r dr\wedge H^4(\Sigma,\mathbb{R}),
\\
\mathcal{K}_{\Lambda^4}(-4) &= r^{-1} dr\wedge H^3(\Sigma,\mathbb{R}) + H^4(\Sigma,\mathbb{R}),
\\
\mathcal{K}_{\Lambda^3}(-3) &= H^3(\Sigma,\mathbb{R}),
\\
\mathcal{K}_{\Lambda^5}(-3) &= r dr\wedge H^4(\Sigma,\mathbb{R}).
\end{align*}
\end{corollary}

\begin{lemma}\label{critical rates for Laplace on 1-forms}
Let $\gamma$ be a harmonic 1-form on the Spin(7)-cone $C(\Sigma)$
homogeneous of rate $\lambda\in (-6,1)$. We have:
\begin{itemize}
\item
If $\lambda\in(-6,0]$, then $\gamma$ must vanish.
\item
If $\lambda\in(0,1)$, then $\gamma = d(\frac{1}{\lambda+1}r^{\lambda+1} \alpha)=d^*(-\frac{1}{\lambda+7}r^{\lambda+2}\beta)$,
where $\alpha\in \Omega^0(\Sigma)$ and $\beta\in\Omega^1(\Sigma)$ are eigenforms of $\Delta_{\ig\Sigma}$ with eigenvalue $(\lambda+1)(\lambda+7)$. In particular, $\gamma$ is closed and co-closed, and therefore an element of $\mathcal{K}_{\Lambda^1}(\lambda)$.
\end{itemize} 
\end{lemma}
\begin{proof}
Following earlier work by Cheeger \cite{cheeger}, Foscolo--Haskins--Nordstr\"om \cite[Theorem A.2]{FHN1} have classified harmonic homogeneous forms of arbitrary pure degree on arbitrary cones. They find four types (i), (ii), (iii) and (iv).
The proof of this Lemma is an application of their classification and the 
Lichnerowicz--Obata Theorem \cite{obata},
which says that the smallest positive eigenvalue of the
scalar Laplace-operator is at least $\frac{\mathrm{Scal}}{6}=7$
if we scale the metric on the link such that the scalar curvature 
equals 42.

Let $\alpha\in \Omega^0(\Sigma)$, $\beta\in\Omega^1(\Sigma)$, and  $\gamma = r^{\lambda}(dr\wedge \alpha + r \beta)$
be a harmonic, homogeneous 1-form of rate $\lambda$.
If $\gamma$ is non-zero and of type (i), (ii) or (iii), then $\alpha$ must be a non-zero eigenfunction of the Laplace-operator with eigenvalue $(\lambda-1)(\lambda+7)$, $(\lambda+1)(\lambda+7)$ and $(\lambda-1)(\lambda+5)$, respectively. The first expression is always negative if $\lambda < 1$, the second expression is at most $7$ if $\lambda\in(-6,0]$, and the third expression is less than $7$ if $\lambda\in(-6,1)$.
By the Lichnerowicz--Obata Theorem $\gamma$ must be of type (ii) and $\lambda\in(0,1)$.
In the latter case $d\alpha = (\lambda+1) \beta$, $d^* \beta = (\lambda+7) \alpha$, and $\gamma$ is of the desired form.

If $\gamma$ is non-zero of type (iv), then $\alpha=0$ and 
$\beta$ is a co-closed, non-zero solution of
$\Delta_{\ig\Sigma}\beta=(\lambda+1)(\lambda+5)\beta$.
By an application of the Bochner formula it follows that the smallest eigenvalue of the Laplacian on $\Sigma$ on co-closed 1-forms is 12,
see \cite[Lemma 2.27]{cheeger1994cone} and \cite[Lemma B.2]{hein2017calabi}.
But $(\lambda+1)(\lambda+5) < 12$ if $\lambda < 1$. Therefore, $\gamma$ must vanish.
\end{proof}

\begin{lemma}
\label{0-1-Dirac}
If $\lambda\in[0,1)$, then
$\mathcal{K}_{\slashed{D}_{-}}(\lambda)\cong \mathcal{K}_{\Lambda^1}(\lambda)$.
\end{lemma}
\begin{proof}
Under the identifications \ref{spin-bundle-iso}
up to constants 
we can write the negative Dirac operator as
\begin{align*}
\slashed{D}_{-}\colon \Gamma(\Lambda^1)\rightarrow \Gamma(\Lambda^0 \oplus \Lambda^2_7),
\quad
\gamma \rightarrow (d^*\gamma, \pi_{7}(d\gamma)).
\end{align*}
The inclusion $\mathcal{K}_{\Lambda^1}(\lambda) \subset\mathcal{K}_{\slashed{D}_{-}}(\lambda)$ follows from the above formula for $\slashed{D}_{-}$ and the reverse inclusion follows from $\Delta = \slashed{D}^2$ and Lemma \ref{critical rates for Laplace on 1-forms}. 
\end{proof}

\begin{lemma}\cite[Proposition 3.3]{KL}
\label{Killing-fields-cone}
Suppose that 
\begin{align*}
Z=r^{\lambda+1} f \partial_r + r^{\lambda} X
\end{align*}
is a Killing vector field on the Spin(7)-cone $(C(\Sigma), \psi_{\ig C})$, where $f$ is a function on $\Sigma$ and $X$ is a vector field on $\Sigma$. 
If $\lambda < 0$,  then $Z$ must vanish.
\end{lemma}

\subsubsection{Harmonic spinors and closed and co-closed forms}

By 
\begin{align*}
\mathcal{H}^k_{\lambda}
=
\{
\gamma\in\Omega^k_{\lambda}|\
d\gamma = 0\ \text{and}\
d^*\gamma = 0
\}
\end{align*}
we denote the space of all closed and co-closed k-forms on $M$ decaying with rate $\lambda$.
$\mathcal{H}^{\mathrm{even}}_{\lambda}$ and $\mathcal{H}^{\mathrm{odd}}_{\lambda}$ are defined analogously.  
Furthermore, we set
\begin{align*}
(\mathcal{H}^k_q)_{\lambda}
=
\{
\gamma\in\mathcal{C}^{\infty}_{\lambda}(\Lambda^k_q)|\
d\gamma = 0\ \text{and}\
d^*\gamma = 0
\}
\end{align*}
if $\Lambda^k_q \subset \Lambda^k$ is a q-dimensional irreducible subrepresentation of $\Lambda^k$.

On compact manifolds any harmonic form is closed and co-closed.
In general, this is not true in the non-compact setting because integration by parts is not always available. However, by Lemma \ref{partial integration} we can use integration by parts if the rate of decay is fast enough.

\begin{lemma}\label{closed and coclosed}
Suppose $\lambda \leq -3$.
\begin{compactenum}[(i)]
\item If $\omega\in L^2_{2,\lambda}(\Lambda^k)$ is harmonic, then $\omega$ is closed and co-closed.
\item $\ker (\slashed{D}_+)_{\lambda}
\cong (\mathcal{H}^4_{1})_{\lambda} \oplus (\mathcal{H}^4_7)_{\lambda}$
and
$\ker (\slashed{D}_-)_{\lambda}
\cong (\mathcal{H}^3_{8})_{\lambda}\cong \mathcal{H}^1_{\lambda}$.
\end{compactenum}
\end{lemma}
\begin{proof}
\textbf{(i):}
If $\lambda \leq -3$ Lemma \ref{partial integration} allows the following integration by parts:
\begin{align*}
0
=
\langle
\Delta \omega,\omega
\rangle_{L^2}
=
\langle 
dd^*\omega,\omega
\rangle_{L^2}
+
\langle
\omega, d^*d\omega
\rangle_{L^2}
=
\|d^*\omega\|_{L^2}^2
+
\|d\omega\|^2_{L^2}.
\end{align*}

\vspace{2mm}

\textbf{(ii):}
By applying Lemma \ref{partial integration} with $\slashed{D}$ as in the proof of (i) we get $\ker(\slashed{D}^2)_{\lambda} = \ker (\slashed{D})_{\lambda}$.
Formula \eqref{spin-bundle-iso} and (i) give
\begin{align*}
&
\ker (\slashed{D}_+)_{\lambda}
=
\ker (\slashed{D}_{-} \slashed{D}_{+})_{\lambda}
\cong 
\ker (\Delta|_{\Lambda^4_{1}})_{\lambda}
\oplus 
\ker (\Delta|_{\Lambda^4_{7}})_{\lambda}
=
(\mathcal{H}^4_1)_{\lambda} \oplus (\mathcal{H}^4_7)_{\lambda}
,
\\
&
\ker (\slashed{D}_-)_{\lambda}
=
\ker (\slashed{D}_{+} \slashed{D}_{-})_{\lambda}
\cong 
\ker (\Delta|_{\Lambda^3_{8}})_{\lambda}
=
(\mathcal{H}^3_8)_{\lambda}
\cong
\mathcal{H}^1_{\lambda}
.
\end{align*}
\end{proof}

\begin{lemma}
\label{individual degree closed and coclosed}
Let $\omega\in L^2_{2,\lambda}(\Lambda^\bullet)$ be closed and coclosed. If $\lambda \leq -3$,
then the individual degree components of $\omega$ are closed and coclosed.
\end{lemma}
\begin{proof}
Denote by $\omega_k$ the degree $k$ component of $\omega$.
The fact that $(d+d^*)\omega=0$ gives $d\omega_k=-d^*\omega_{k+2}$.
The condition $\lambda \leq -3$ allows the following integration by parts:
\begin{align*}
\|d\omega_k\|^2_{L^2}
=
\langle
d\omega_k,d\omega_k
\rangle_{L^2}
=
-
\langle
d\omega_k,d^* \omega_{k+2}
\rangle_{L^2}
=
-
\langle
\omega, d^*d^* \omega_{k+2}
\rangle_{L^2}
=0. 
\end{align*}
\end{proof}

\begin{lemma}
\label{no-parallel-spinors}
If $\lambda \leq -3$, then $\ker(\slashed{D})_{\lambda} = \{0\}$.
\end{lemma}
\begin{proof}
We use the Lichnerowicz formula
\begin{align*}
\slashed{D}^2 = \nabla^* \nabla + \frac{1}{4}\mathrm{scal}(g).
\end{align*}
Because $g$ is Ricci-flat, the scalar curvature vanishes and the Dirac Laplacian coincides with the rough Laplacian. If $s\in L^2_{k,\lambda}(\mathbf{S})$ for $\lambda \leq -3$, we can apply Lemma \ref{partial integration} to obtain
\begin{align*}
\langle \slashed{D}^2 s, s\rangle_{L^2}
=
\langle \nabla^* \nabla s, s\rangle_{L^2}
=
\| \nabla s \|_{L^2}.
\end{align*}
Therefore, $s$ is parallel if $s\in \ker(\slashed{D})_{\lambda}$. In particular, its point-wise norm is constant on $M$.
Because the $L^2_{k,\lambda}$-norm of $s$ is finite, $s$ must vanish.
\end{proof}

The spin bundle identification \eqref{spin-bundle-iso}, Lemma \ref{closed and coclosed} (ii) and Lemma \ref{no-parallel-spinors} imply

\begin{corollary}
\label{no-harmonic-spinors}
$(\mathcal{H}^4_1)_{\lambda}$, $(\mathcal{H}^4_7)_{\lambda}$, $(\mathcal{H}^3_8)_{\lambda}$ and $\mathcal{H}^1_{\lambda}$ are zero if $\lambda \leq -3$.
\end{corollary}

For 1-forms this statement can be improved:

\begin{lemma}
\label{vanishing-harmonic-forms}
A harmonic 1-form $\gamma\in\mathcal{C}_{\lambda}^{\infty}(T^*M)$
vanishes if $\lambda \leq 0$.
\end{lemma}
\begin{proof}
The statement is true for $\lambda \leq -3$ by Corollary \ref{no-harmonic-spinors}.
By Theorem \ref{constant away from crit rates} the kernel of the Laplace operator acting on 1-forms can only change at critical rates. However, by Lemma \ref{critical rates for Laplace on 1-forms} there are no critical rates in the interval $[-6,0]$. 
\end{proof}

\begin{lemma}\label{Injectivity of negative Dirac-operator}
The negative Dirac operator
\begin{align*}
(\slashed{D}_{-})_{l+1,\lambda+1}
\colon
L^2_{l+1,\lambda+1}(\Lambda^3_{8})
\rightarrow
L^2_{l,\lambda}(\Lambda^4_{1}\oplus \Lambda^4_7).
\end{align*}
is injective if $\lambda \leq -1$ and surjective if $\lambda \geq -5$.
\end{lemma}
\begin{proof}
Under the identification \eqref{spin-bundle-iso} the statement about injectivity follows from Lemma \ref{vanishing-harmonic-forms}. 
The adjoint of $(\slashed{D}_{-})_{l+1,\lambda+1}$ is the positive Dirac operator
\begin{align*}
(\slashed{D}_+)_{m+1,-8-\lambda}\colon L^2_{m+1,-8-\lambda}(\Lambda^4_1 \oplus \Lambda^4_7) \rightarrow L^2_{m,-9-\lambda}(\Lambda^3_8).
\end{align*}
Then $\mathrm{Coker}(\slashed{D}_{-})_{\lambda+1}=\ker (\slashed{D}_+)_{-8-\lambda}$ and with Lemma \ref{no-parallel-spinors} the cokernel is zero if $\lambda \geq -5$.
\end{proof}

We will need the following analogues of the Hodge decomposition theorem on compact manifolds:

\begin{proposition}
\cite[Proposition 4.33]{KL}
\label{Hodge-decomposition}
Suppose $\lambda+1$ is a non-critical rate for $d+d^*$. Let $0 \leq k \leq 8$. If $\lambda > -4$,  we have
\begin{align*}
\Omega_{l,\lambda}^k
=
d(\Omega^{k-1}_{l+1,\lambda+1})+d^*(\Omega^{k+1}_{l+1,\lambda+1})
\oplus
\mathcal{H}_{-8-\lambda}^k.
\end{align*}
\end{proposition}

\begin{proposition}
\cite[Proposition 4.31, Corollary 4.32]{KL}
\label{Hodge-decomposition-L2}
Suppose $\lambda+1$ is a non-critical rate for $d+d^*$. Let $0 \leq k \leq 8$. If $\lambda < -4$,  we have
an $L^2$-orthogonal decomposition
\begin{align*}
\Omega_{l,\lambda}^k
=
d(\Omega^{k-1}_{l+1,\lambda+1}) \oplus d^*(\Omega^{k+1}_{l+1,\lambda+1})
\oplus
\mathcal{H}_{\lambda}^k
\oplus
W^k_{l,\lambda}
,
\end{align*}
where $W^k_{l,\lambda}$ is isomorphic to $\mathcal{H}^k_{-8-\lambda}/\mathcal{H}^k_{\lambda}$.
\end{proposition}

In the next definition we define the main differential operator involved in studying the moduli space of AC Spin(7)-manifolds.

\begin{definition}
\label{main operator}
We denote the exterior derivative restricted to sections of $\Lambda^4_{35}=\Lambda^4_{\mathrm{ASD}}$ by 
\begin{align*}
d_{\mathrm{ASD}}\colon
\Omega^4_{35}(M)
\rightarrow
d\Omega^4(M),
\quad
\gamma \mapsto d\gamma.
\end{align*}
By $(d_{\mathrm{ASD}})_{l,\nu}$ we denote its continuous extension
\begin{align*}
(d_{\mathrm{ASD}})_{l,\nu}\colon L^2_{l,\nu}(\Lambda^4_{35})\rightarrow d(\Omega^4_{l,\nu}).
\end{align*}
\end{definition}

\begin{lemma}
\label{surjectivity of main operator}
Suppose that $\nu+1$ is a non-critical rate of $d+d^*$.
If $\nu > -4$, the operator $(d_{\mathrm{ASD}})_{l,\nu}$ is surjective.
\end{lemma}
\begin{proof}
We can prove this by using the Hodge decomposition from Proposition \ref{Hodge-decomposition}. Let $\alpha\in d(\Omega^4_{l,\nu})$ be exact. Then we can write $\alpha = d\eta$ for some co-exact $\eta \in d^*(\Omega^5_{l+1,\nu+1})$.
But then $\alpha=d(\eta-*\eta)$ is the exterior derivative of an anti-self-dual form and hence $(d_{\mathrm{ASD}})_{l,\nu}$ is surjective.
\end{proof}

\begin{remark}
\label{explanation-for-obstructions}
The main reason why we have to restrict to rates $\nu > -4$ in  Theorem \ref{main-thm} is that in the $L^2$-setting we cannot even expect that 
\begin{align*}
d\colon L^2_{l,\nu}(\Lambda^4_{1}\oplus \Lambda^4_{7} \oplus \Lambda^4_{35}) \rightarrow d(\Omega^4_{l,\nu})
\end{align*}
is surjective.
The reason is that in the Hodge decomposition from Proposition \ref{Hodge-decomposition-L2} forms in the space $W^4_{l,\nu}$ are not necessarily closed. If we denote by $(W_c)^4_{l,\nu}$ the subspace of closed forms in $W^4_{l,\nu}$ and by $(W_{\perp})^4_{l,\nu}$ its $L^2$-orthogonal complement in $W^4_{l,\nu}$, then every form in $d(\Omega^4_{l,\nu})$ can be written as 
\begin{align*}
d(\eta+\omega)
\end{align*}
for unique $\eta\in d^*(\Omega^5_{l+1,\nu+1})$ and $\omega\in(W_{\perp})^4_{l,\nu}$. If $\omega \neq 0$, then we cannot use the same trick as in the proof of Lemma \ref{surjectivity of main operator}.
\end{remark}

\begin{lemma}
\label{image-Banach}
Suppose $\nu$ is a non-critical rate of the operator $d+d^*$.
Then 
$d\Omega^4_{l,\nu}(M)$ is a closed subspace of $\Omega^5_{l-1,\nu-1}$,
and therefore a Banach space.
\end{lemma}
\begin{proof}
If $\nu \leq -4$, this is true because all components in the $L^2$-orthogonal decomposition from Proposition \ref{Hodge-decomposition-L2} are closed.

Next we consider the case $\nu > -4$.
Let $\{ d\gamma_j \}_{j\in\mathbb{N}}$ be a Cauchy sequence in $d\Omega^4_{l,\nu}(M)$. By Lemma \ref{surjectivity of main operator} we can assume that $\gamma_j\in L^2_{l,\nu}(\Lambda^4_{35})$ for all $j\in\mathbb{N}$. 
Because $\gamma_j$ is anti-self dual we have $d^*\gamma_j=*d\gamma_j$ and, therefore, $\{(d+d^*)\gamma_j\}_{j\in\mathbb{N}}$ is a Cauchy sequence in $\Omega^{\bullet}_{l-1,\nu-1}$. By Proposition \ref{poincare inequality}
there exists $\gamma\in\Omega^4_{l,\nu}$ such that
$(d+d^*)\gamma = \lim_{j\rightarrow \infty}(d+d^*)\gamma_j$.
This finishes the proof.
\end{proof}

\subsection{Cohomology of AC Spin(7)-manifolds}

\label{section-topology}

Suppose $(M,\psi)$ is an AC Spin(7)-manifold.
The compactly supported cohomology groups $H^k_{\mathrm{cs}}(M,\mathbb{R})$ of $M$ are the cohomology groups associated to the chain complex of compactly supported forms on $M$. Any representative of a class in $H^k_{\mathrm{cs}}(M,\mathbb{R})$ is a closed $k$-form and therefore induces a class in $H^k(M,\mathbb{R})$. This is well-defined at the cohomology level and induces a map
\begin{align*}
\mathcal{I}^k:H^k_{\mathrm{cs}}(M,\mathbb{R})\rightarrow H^k(M,\mathbb{R}).
\end{align*} 
It follows straight from Definition \ref{def-AC-Spin(7)} 
that if $r > R$, there is an embedding $\iota_r:\Sigma \rightarrow M$ given by $\iota_r = F(r, \cdot )$.
This induces a restriction map $\iota_r^*: H^k(M,\mathbb{R})\rightarrow H^k(\Sigma,\mathbb{R})$. Because the embeddings are homotopic for different values of $r$, the map $\iota_r^*$ does not depend on $r$. Henceforth we will denote it by 
\begin{align}
\label{restriction map}
\Upsilon^k:H^k(M,\mathbb{R})\rightarrow H^k(\Sigma,\mathbb{R}).
\end{align}
The maps $\mathcal{I}^k$ and $\Upsilon^k$ are part of a long exact sequence given by
\begin{align}
\label{long-exact-sequence}
\cdots
\rightarrow
H^k_{\mathrm{cs}}(M,\mathbb{R})
\xrightarrow{\mathcal{I}^k}
H^k(M,\mathbb{R})
\xrightarrow{\Upsilon^k}
H^k(\Sigma,\mathbb{R})
\xrightarrow{\partial^k}
H^{k+1}_{\mathrm{cs}}(M,\mathbb{R})
\rightarrow
\cdots
\end{align}
The boundary map $\partial^k$ can be described as follows. If $[\alpha]\in H^k(\Sigma,\mathbb{R})$ set $\partial^k[\alpha] := [d(\chi \alpha)]$, where $\chi$ is the cut-off function from Definition \ref{def-radial-function}. This is well-defined. Note that even though the form $d(\chi \alpha)$ is exact, the map $\partial^k$ is non-trivial because $d(\chi \alpha)$ in general cannot be written as the exterior derivative of a compactly supported form. 

For us it is important to have a good description of preimages of cohomology classes in $\im \Upsilon^k$. We say that a $k$-form $\gamma$ on $M$ is \textit{translation invariant}
if there exists $R' \geq R$ such that for all $r \geq R'$
we have
\begin{gather*}
\iota_r^*(\gamma) = \iota_{R'}^*(\gamma),
\\
\iota_r^*(\partial_r \lrcorner \gamma)
=
\iota_{R'}^*(\partial_r \lrcorner \gamma).
\end{gather*}
Equivalently there exist $R'\geq R$, $\alpha\in\Omega^{k-1}(\Sigma)$ and $\beta\in\Omega^k(\Sigma)$ such that on 
$F((R',\infty)\times\Sigma)$ we have $\gamma = dr\wedge\alpha+\beta$. If in addition $\alpha = 0$, we say that $\gamma$ is a \textit{lift}.

\begin{lemma}
\cite[Corollary 5.9]{Marshall}
\label{lift}
Let $[\beta]\in \im \Upsilon^k$ where $\beta$ is any representative. Then
a preimage of $[\beta]$ under $\Upsilon^k$ can be represented by a lift, i.e.  there exists $\zeta\in\Omega^k_{\mathrm{cs}}(M)$ such that $\xi = \chi \beta + \zeta $ is closed and $\Upsilon^k[\xi]=[\beta]$.
\end{lemma}

\begin{lemma}
\label{orthogonality of image}
$\im \Upsilon^3$ and $\im \Upsilon^4$ annihilate each other under the Poincar\'{e} pairing, i.e. with respect to harmonic representatives we have
\begin{align*}
*\im \Upsilon^4 \perp_{L^2} \im\Upsilon^3
\quad
\text{and}
\quad
*\im \Upsilon^3\perp_{L^2} \im \Upsilon^4.
\end{align*}
\end{lemma}
\begin{proof}
Let $\alpha$ be a harmonic representative of a class
$[\alpha]\in \im\Upsilon^3$
and $\beta$ a harmonic representative of a class
$[\beta]\in\im \Upsilon^4$.
Then by Lemma \ref{lift} there exist a closed 3-form $\gamma$, a closed 4-form $\eta$ and compactly supported forms $\gamma_{-}$ and $\eta_{-}$ such that $\gamma = \chi \alpha + \gamma_{-}$
and $\eta=\chi \beta+\eta_{-}$.
Stokes' theorem gives
\begin{align*}
0 = \int_M d(\gamma\wedge\eta)
=
\lim_{r\rightarrow\infty}
\int_{\{r\}\times \Sigma}
(\gamma|_{\{r\}\times \Sigma})
\wedge
(\eta|_{\{r\}\times \Sigma})
=
\lim_{r\rightarrow\infty}
r^{7} \int_{\Sigma} \alpha\wedge\beta
\\
+
\lim_{r\rightarrow\infty}
\int_{\{r\}\times \Sigma}
\Big(
\alpha\wedge(\eta_{-}|_{\{r\}\times \Sigma})
+
(\gamma_{-}|_{\{r\}\times \Sigma})\wedge\beta
+
(\gamma_{-}|_{\{r\}\times \Sigma})
\wedge
(\eta_{-}|_{\{r\}\times \Sigma})
\Big)
\end{align*}
Because the integrand in the second limit is compactly supported this limit is zero. Therefore, we get
\begin{align*}
\langle *\alpha,\beta\rangle_{L^2}
=
\langle \alpha,*\beta\rangle_{L^2}
=
\int_{\Sigma} \alpha\wedge\beta = 0.
\end{align*}
\end{proof}

The reason that topology is relevant for us is that we need to understand closed and co-closed forms decaying with the $L^2$-rate $-4$. The following Proposition relates them to the cohomology groups of $M$. The result is due to Lockhart \cite{Lockhart-Hodge}. We will use a version adapted to the AC setting \cite[Theorem 6.5.2]{Jason-PhD}. 
\begin{proposition}
\label{Hodge Theorem}
We have
\begin{align*}
\mathcal{H}^k_{L^2}
=
\mathcal{H}^k_{-4}
\cong
\begin{cases}
H^k(M,\mathbb{R})\quad\quad & k > 4,
\\
\mathcal{I}^4(H^4_{\mathrm{cs}}(M,\mathbb{R})) 
\quad\quad & k = 4,
\\
H^k_{\mathrm{cs}}(M,\mathbb{R})\quad\quad & k < 4. 
\end{cases}
\end{align*}
\end{proposition}

Because we study deformations of Spin(7)-structures our main interest are 4-forms. By Proposition \ref{Hodge Theorem} harmonic 4-forms which lie in $L^2$ can be understood to be  purely topological. More specifically we have 
\begin{align*}
\mathcal{H}^4_{L^2} 
\cong 
\mathcal{I}^4(H^4_{\mathrm{cs}}(M,\mathbb{R}))
\subset
H^4(M,\mathbb{R}).
\end{align*}
Splitting up into self dual and anti-self-dual 4-forms  induces the decomposition
\begin{align}
\mathcal{I}^4(H^4_{\mathrm{cs}}(M,\mathbb{R}))
=
(\mathcal{H}^4_{+})_{L^2}
\oplus
(\mathcal{H}^4_{-})_{L^2}. 
\label{decomposition-L2 cohomology}
\end{align}
This splitting can also be understood in a topological way.
Let $[\xi],[\eta]\in \mathcal{I}^4(H^4_{\mathrm{cs}}(M,\mathbb{R}))$,
where $\xi$ and $\eta$ are compactly supported representatives of the corresponding cohomology classes. Then
\begin{align*}
\int_M \xi\wedge \eta
\end{align*}
is finite and defines a symmetric bilinear form on $\mathcal{I}^4(H^4_{\mathrm{cs}}(M,\mathbb{R}))$. 
To see that this is well-defined suppose that  $\xi'$ is another compactly supported representative of the cohomology class of $\xi$.
Writing $[\xi]_{\mathrm{cs}}, [\xi']_{\mathrm{cs}}$ for the corresponding classes in $H^4_{\mathrm{cs}}(M,\mathbb{R})$, we get
\begin{align*}
\mathcal{I}^4([\xi]_{\mathrm{cs}}-[\xi']_{\mathrm{cs}})=0.
\end{align*}
By the exactness of \eqref{long-exact-sequence} and the description of the boundary map $\partial^3$, there exist $[\alpha]\in H^3(\Sigma)$ and $\gamma\in \Omega^3_{\mathrm{cs}}(M)$ such that 
\begin{align*}
\xi' = \xi + d(\chi \alpha + \gamma).
\end{align*}
Setting $\eta'= \eta + d\phi$, we have
\begin{align*}
\int_M d(\chi \alpha + \gamma)\wedge \eta'
=
\int_M d((\chi \alpha+\gamma)\wedge(\eta+d\phi))
=
\lim_{r\rightarrow \infty} 
\int_M
\alpha \wedge d(\phi|_{\{r\}\times \Sigma})
=
0.
\end{align*}
This bilinear form is non-degenerate because
the pairing of $H^4_{\mathrm{cs}}(M,\mathbb{R})$
and $H^4(M,\mathbb{R})$ is non-degenerate.
$(\mathcal{H}^4_{+})_{L^2}$ is a positive definite subspace of $\mathcal{I}^4(H^4_{\mathrm{cs}}(M,\mathbb{R}))$ with respect to this bilinear form and $(\mathcal{H}^4_{-})_{L^2}$ is a negative definite subspace.

The last topological ingredient we need is that sufficiently fast decaying forms are exact on the end.

\begin{lemma}
\label{fast decay -> exact on end}
\cite[Lemma 2.12]{Karigiannis-desing}
Let $\gamma$ be a smooth k-form on the cone $C(\Sigma)$.
If 
\begin{align*}
|\nabla_{\ig C}^j \gamma|_{g_C}
=
\mathcal{O}(r^{\lambda-j})
\
\text{as}\ r\rightarrow\infty
\,
\text{for all}\
j\in\mathbb{N},\
\text{for some}\
\lambda < -k,
\end{align*} 
then there exists a smooth $k-1$ form 
$\xi$ on $(R,\infty)\times\Sigma$ such that $d\xi = \gamma$. 
\end{lemma}

\section{The Moduli Space is an orbifold}

\label{smoothness-mod-space}

In this section we consider the moduli space of torsion-free AC Spin(7)-structures of rate $\nu$ on the manifold $M$.
As explained in the introduction we do not want to consider deformations of the Spin(7)-cone or, equivalently, of the link $\Sigma$ because deformations of nearly parallel $\mathrm{G}_2$-manifolds are not very well understood. Therefore, we fix an asymptotic Spin(7)-cone $C:=(C(\Sigma),\psi_{\ig C})$.
The diffeomorphism group of $M$ acts on the set of AC Spin(7)-structures asymptotic to $C$ at a fixed rate. 
Indeed, if $\Phi$ is a diffeomorphism of $M$
and $\psi$ is asymptotic to $\psi_{\ig C}$ at rate $\nu$ with respect to some identification $F$
of the cone and $M$ outside a compact subset as in Definition \ref{def-AC-Spin(7)}, then $\Phi^*\psi$ is an AC Spin(7)-structure on $M$ asymptotic to $C$ at rate $\nu$ with respect to $F':= \Phi^{-1} \circ F$. Because $\Phi^*\psi$ does not decay to $\psi_{\ig C}$ at rate $\nu$ with respect to $F$
unless $\Phi$ decays sufficiently fast to an automorphism of $C$, we can break the diffeomorphism invariance by fixing $F$ and taking the quotient by suitably decaying diffeomorphisms.
For the sake of simplicity, we only quotient by diffeomorphisms which decay to the identity. Those decaying to some automorphism of $C$ can in principle be divided out later.
The above procedure also normalises a scale.
For $\lambda > 0$ the rescaled Spin(7)-structure $\lambda^4 \psi$
decays to $\psi_{\ig C}$ only after composing $F$ with the diffeomorphism $(r,x)\mapsto (\lambda r,x)$ of the cone.
Our results do not depend on the choice of $F$.

For the description of the moduli space it is convenient to choose a reference point.
Let $(M,\psi,g)$ be an AC Spin(7)-manifold at rate $\nu < 0$ with respect to $F$. Let $\mathcal{A}_{\nu}$ be the space of admissible 4-forms on $M$ which decay with the same rate as $\psi$ in the chosen gauge, i.e.
\begin{align}
\mathcal{A}_{\nu}
:=
\{ \tilde{\psi}\in\mathcal{A}(M)|\
\tilde{\psi}-\psi \in \mathcal{C}^{\infty}_{\nu}(\Lambda^4 T^*M)
\}
\subset \psi + \mathcal{C}^{\infty}_{\nu}(\Lambda^4 T^*M)
.
\label{decaying-admissible-4-forms}
\end{align}
The space of all (up to the choice of $F$) torsion-free AC Spin(7)-structures on $M$ asymptotic to $C$ at rate $\nu$ is denoted by
\begin{align*}
\mathcal{X}_{\nu}
:=
\{\tilde{\psi}\in \mathcal{A}_{\nu}\, |\, d\tilde{\psi}=0 \}.
\end{align*}
Denote by $\mathcal{D}_{\lambda}$ the group of diffeomorphisms generated by vector fields in $\mathcal{C}^{\infty}_{\lambda}(TM)$.
The group $\mathcal{D}_{\nu+1}$ acts on $\mathcal{A}_{\nu}$
and $\mathcal{X}_{\nu}$ by pull-back.
Then
$\mathcal{M}_{\nu}:=\mathcal{X}_{\nu}/\mathcal{D}_{\nu+1}$ is the moduli space of torsion-free AC Spin(7)-structures on $M$ with decay rate $\nu$ asymptotic to $C$. 
We want to use the implicit function theorem for smooth maps between Banach spaces to  prove that $\mathcal{M}_{\nu}$ is an orbifold for particular rates $\nu$. Therefore, we equip
$\mathcal{A}_{\nu}$ and $\mathcal{D}_{\nu+1}$ with the $L^2_{l,\nu}(\Lambda^4T^*M)$ and $L^2_{l+1,\nu+1}(TM)$ topologies rather than the Frechet space topology of smooth forms and vector fields.
We choose some $l \geq 6$, so that by Theorem \ref{Sobolev embedding} (i) we have an embedding $L^2_{l,\nu} \hookrightarrow \mathcal{C}^{1,\alpha}_{\nu}$.
The action of $\mathcal{D}_{\nu+1}$ on $\mathcal{A}_{\nu}$ is continuous, 
$\mathcal{X}_{\nu}$ carries the subspace topology of $\mathcal{A}_{\nu}$, and $\mathcal{M}_{\nu}$ the quotient topology, with respect to which we want to prove the smooth manifold structure. As auxiliary objects we introduce 
$\mathcal{A}_{l,\nu}$ and $\mathcal{D}_{l+1,\nu+1}$, the completions of $\mathcal{A}_{\nu}$ and $\mathcal{D}_{\nu+1}$, respectively.

\subsection{The space of torsion-free Spin(7)-structures and the stabiliser} 

Before we treat the moduli space $\mathcal{M}_{\nu}$, we first 
study the space of torsion-free Spin(7)-structures $\mathcal{X}_{l,\nu}$,
which are $L^2_{l,\nu}$-regular, and the stabiliser
\begin{align*}
\mathcal{I}_{\psi}:= \{\Phi\in\mathcal{D}_{l+1,\nu+1}\, |\, \Phi^*\psi=\psi\}
\end{align*}
of $\psi$ in $\mathcal{D}_{l+1,\nu+1}$. 
Because isometries of smooth Riemannian metrics are smooth by a result of
Myers--Steenrod \cite{smooth-isometries}, $\mathcal{I}_{\psi}$ can alternatively be defined as the stabiliser of $\psi$ in $\mathcal{D}_{\nu+1}$.
Using the implicit function theorem, we show that $\mathcal{X}_{l,\nu}$ is a smooth manifold under a suitable assumption on the rate $\nu$.

\begin{proposition}
\label{tangent-space-torsion-free}
Suppose $\nu > -4$, and that $\nu$ and $\nu+1$ are non-critical rates of the operator $d+d^*$. Then
$\mathcal{X}_{l,\nu}$ is a smooth manifold
and the tangent space $T_{\psi}\mathcal{X}_{l,\nu}$ is given by
the kernel of the linear map
\begin{align}
\label{linearisation}
d\colon T_{\psi}\mathcal{A}_{l,\nu} = L^2_{l,\nu}(\Lambda^4_1 \oplus \Lambda^4_7 \oplus \Lambda^4_{35})
\rightarrow d\Omega^4_{l,\nu}.
\end{align}
\end{proposition}
\begin{proof}
$\mathcal{X}_{l,\nu}$ is the zero level set of the exterior derivative
\begin{align}
\label{zero-level-set-map}
d\colon \mathcal{A}_{l,\nu} \rightarrow d\Omega^4_{l,\nu}.
\end{align}
By Lemma \ref{image-Banach} this is a smooth map between Banach manifolds because $\nu$ is not a critical rate for $d+d^*$.
By \eqref{tangent-space-admissible-forms} we know that 
$T_{\psi}\mathcal{A}_{l,\nu}=L^2_{l,\nu}(\Lambda^4_{1} \oplus \Lambda^4_7 \oplus \Lambda^4_{35})$.
The linearisation of \eqref{zero-level-set-map} at $\psi$
is 
the map \eqref{linearisation}.
Under our assumptions this map is surjective by Lemma \ref{surjectivity of main operator}. 
The statement follows from the implicit function theorem for smooth maps between Banach spaces.
\end{proof}

\begin{remark}
By Remark \ref{explanation-for-obstructions} in the $L^2$-setting the linearisation of \eqref{zero-level-set-map} cannot be expected to be surjective if $(W_{\perp})^4_{l,\nu}$ is non-trivial. Therefore, the space $\mathcal{X}_{l,\nu}$ is in general not smooth for $L^2$-rates $\nu < -4$.
\end{remark}

Next we show that $\mathcal{I}_{\psi}$ is finite. To prove this, it is enough to show that $\psi$ cannot have any continuous symmetries in $\mathcal{D}_{\nu+1}$, i.e. that Killing vector fields of rate $\nu+1$ vanish. This implies that the stabiliser $\mathcal{I}_{\psi}$ is discrete, and thus finite because it is also compact.

\begin{proposition}
\label{no-Killing}
Suppose $\nu<0$. Let $\xi\in L^2_{l,\nu+1}(TM)$ be a Killing vector field for $g$, i.e. $\mathcal{L}_{\xi}g=0$. Then $\xi$ vanishes.
Furthermore, the exterior derivative is injective
on $L^2_{l,\nu}(\Lambda^3_8)$.
\end{proposition}
\begin{proof}
Because $(M,\psi,g)$ is Ricci-flat, any Killing vector field is harmonic. 
If $\nu \leq -1$, then $\xi^{\flat}$ is a harmonic 1-form of non-positive decay rate and hence has to vanish by Lemma \ref{vanishing-harmonic-forms}.
If $\xi\lrcorner\psi\in L^2_{l,\nu}(\Lambda^3_8)$ is closed,
then $\xi$ is a Killing vector field. The statement for $\nu \leq -1$ follows.

It is left to consider the case $\nu\in(-1,0)$. 
By Theorem \ref{Lockhart-McOwen-key} there exists a critical rate $\lambda+1 < \nu+1$ for the Laplace operator such that 
\begin{align}
\label{str-Killing-field}
\xi = \chi Z + \mathcal{O}(r^{\lambda+1-\varepsilon}),
\end{align}
where 
\begin{align*}
Z=r^{\lambda+1} f \partial_r + r^{\lambda} X.
\end{align*}
Here $f$ is a function on $\Sigma$ and $X$ a vector field on $\Sigma$,
and $Z$
is $g_{\ig C}$-dual to a harmonic 1-form on $C(\Sigma)$ homogeneous of order $\lambda+1$. 

Our goal is to show that $Z$ is a Killing vector field for $g_{\ig C}$.
Then by Lemma \ref{Killing-fields-cone} $Z$ must vanish.
Repeating the argument for critical rates $\lambda+1\in (0,1)$
shows that $\xi$ is a Killing vector field for $g$ with non-positive decay rate, and therefore vanishes as in the case $\nu \leq -1$.
For general vector fields $X, Y, V$ on the cone we have
\begin{align*}
\mathcal{L}_V g(X,Y) &= g(\nabla_X V,Y) + g(X, \nabla_Y V),
\\
\mathcal{L}_V g_{\ig C}(X,Y) &= g_{\ig C}(\nabla^{\ig C}_X V,Y) + g_{\ig C}(X,\nabla^{\ig C}_Y V),
\end{align*}
and therefore
\begin{align*}
(\mathcal{L}_V g-\mathcal{L}_V g_{\ig C})(X,Y)
=\,
&(g-g_{\ig C})(\nabla_X V, Y) + (g-g_{\ig C})(X, \nabla_Y V)
\\
&+ g_{\ig C}((\nabla_X-\nabla^{\ig C}_X)V,Y)
+
g_{\ig C}(X,(\nabla_Y-\nabla^{\ig C}_Y)V)
\end{align*}
We have $|g-g_{\ig C}| = \mathcal{O}(r^{\nu})$, $|\nabla-\nabla^{\ig C}|=\mathcal{O}(r^{\nu-1})$ as $g$ is AC with rate $\nu$, and
$|\xi|=\mathcal{O}(r^{\lambda+1})$, $|\nabla \xi|=\mathcal{O}(r^{\lambda})$
by \eqref{str-Killing-field} and elliptic regularity.
Therefore, we get
\begin{align*}
\mathcal{L}_{\xi} g_{\ig C}=
\mathcal{L}_{\xi}g_{\ig C}-\mathcal{L}_{\xi}g = \mathcal{O}(r^{\nu+\lambda})
=
\mathcal{O}(r^{\lambda-\varepsilon})
\end{align*}
for some $\varepsilon > 0$.
On the other hand we have
\begin{align*}
\mathcal{L}_{\xi} g_{\ig C}
=
\mathcal{L}_Z g_{\ig C} + \mathcal{O}(r^{\lambda-\varepsilon})
\end{align*}
and $\mathcal{L}_Z g_{\ig C}$ is homogeneous of rate $\lambda$. Therefore, $\mathcal{L}_Z g_{\ig C} = 0$.
\end{proof}

\subsection{Slice construction for moduli space}

Now we are ready to study the moduli space $\mathcal{M}_{\nu} = \mathcal{X}_{\nu}/\mathcal{D}_{\nu+1}$. We want to break the action of $\mathcal{D}_{\nu+1}$ on $\mathcal{A}_{\nu}$ and in each orbit close to $\psi$ choose in a smooth fashion a representative which is unique up the the action of the stabiliser
$\mathcal{I}_{\psi}$. 
Ebin \cite{ebin1970manifold} showed how to find a slice for the diffeomorphism action on the space of Riemannian metrics. 
Using Proposition \ref{tangent-space-torsion-free} and simplifications of Ebin's approach in our setting, which were explained by Nordstr\"om \cite{Nordstrom-PhD}, 
we find that a good slice $\mathcal{S}_{\psi}$ around $\psi$ needs to satisfy three properties:

\begin{theorem}
\label{good-slices}
\cite[Section 3.1.3]{Nordstrom-PhD}
Let $K$ be a closed subspace which is a complement of $T_{\psi}(\mathcal{D}_{l+1,\nu+1}\cdot \psi)$ in $T_{\psi}\mathcal{A}_{l,\nu}$.
Let $\mathcal{S}_{\psi}$ be a smooth submanifold of $\mathcal{A}_{l,\nu}$ which contains $\psi$ and satisfies 
\begin{compactenum}[(S.1)]
\addtolength{\itemindent}{.5cm}
\item $T_{\psi}\mathcal{S}_{\psi} = K$,
\item $\mathcal{S}_{\psi}$ is $\mathcal{I}_{\psi}$-invariant,
\item all $\tilde{\psi}\in\mathcal{R}_{\psi}:= \mathcal{S}_{\psi} \cap \mathcal{X}_{l,\nu}$ are smooth 4-forms on $M$.
\end{compactenum}
Then we have:
\begin{compactenum}[(i)]
\item $\mathcal{R}_{\psi}$ is a smooth manifold,
\item the map 
\begin{align*}
\mathcal{R}_{\psi} \subset \mathcal{X}_{\nu} \rightarrow \mathcal{M}_{\nu}=\mathcal{X}_{\nu}/\mathcal{D}_{\nu+1},
\quad
\tilde{\psi} \mapsto \tilde{\psi} \mathcal{D}_{\nu+1}
\end{align*}
is open,
\item the induced map $\mathcal{R}_{\psi}/\mathcal{I}_{\psi} \rightarrow \mathcal{M}_{\nu}$ is a homeomorphism onto its image,
\item 
the map
\begin{align}
\label{submersion}
\mathcal{D}_{l+1,\nu+1}\times \mathcal{R}_{\psi} \rightarrow \mathcal{X}_{l,\nu}
\end{align}  
is a smooth submersion onto a neighbourhood of $\psi$.
\end{compactenum}
\end{theorem}

Once we find $\mathcal{S}_{\psi}$, we can use Theorem \ref{good-slices} (i)-(iv) to conclude that  under the assumptions on $\nu$ in Proposition \ref{tangent-space-torsion-free} $\mathcal{M}_{\nu}$ is an orbifold: around $\psi$ we can use $\mathcal{R}_{\psi}/\mathcal{I}_{\psi}$ as a chart, and the transition to another such chart $\mathcal{R}_{\tilde{\psi}}/\mathcal{\tilde{\psi}}$ centred at $\tilde{\psi}\in\mathcal{X}_{\nu}$ can be described via a section of the submersion \eqref{submersion}. 
If $\mathcal{I}_{\psi}$ is trivial or acts trivially on $\mathcal{R}_{\psi}$, we can strengthen our conclusion and find that $\mathcal{M}_{\nu}$ is smooth in a neighbourhood of $\psi$. One particular way to check if the action of $\mathcal{I}_{\psi}$ on $\mathcal{R}_{\psi}$ is trivial is to look at the projection $\mathcal{R}_{\psi}\rightarrow H^4(M,\mathbb{R})$ to the cohomology group. This is well-defined because elements in $\mathcal{R}_{\psi}$ are smooth, closed 4-forms. Because elements in $\mathcal{I}_{\psi}$ are isotopic to the identity, the projection is $\mathcal{I}_{\psi}$-invariant. If $\mathcal{R}_{\psi}\rightarrow H^4(M,\mathbb{R})$ is an embedding, all forms in $\mathcal{R}_{\psi}$ therefore represent different points in the moduli space and we can use $\mathcal{R}_{\psi}$ as a smooth chart.

To wrap up our discussion of the moduli space $\mathcal{M}_{\nu}$, we are left to find a good slice $\mathcal{S}_{\psi}$ as in Theorem \ref{good-slices}.
To do so, we first determine a complement of $T_{\psi}(\mathcal{D}_{l+1,\nu+1}\cdot\psi)$ in $T_{\psi}\mathcal{A}_{l,\nu}=L^2_{l,\nu}(\Lambda^4_{1} \oplus \Lambda^4_7 \oplus \Lambda^4_{35})$. 
To compute $T_{\psi}(\mathcal{D}_{l+1,\nu+1}\cdot \psi)$, let  $F_t$ be the 1-parameter subgroup of diffeomorphisms generated by some $X\in\mathcal{C}^{\infty}_{\nu+1}(TM)$. Then
\begin{align*}
\left.
\frac{d}{dt}
\right|_{t=0}
F_t^*\psi
=
\mathcal{L}_X \psi
=
X\lrcorner d\psi
+
d(X\lrcorner\psi)
=
d(X\lrcorner\psi).
\end{align*}
By \eqref{3-forms-of-type-8} we get $T_{\psi}(\mathcal{D}_{\nu+1}\cdot \psi)=d(\mathcal{C}^{\infty}_{\nu+1}(\Lambda^3_8))$. 
Analogously $T_{\psi}(\mathcal{D}_{l+1,\nu+1}\cdot \psi)=d (L^2_{l+1,\nu+1}(\Lambda^3_8))$. In particular, $T_{\psi}(\mathcal{D}_{l+1,\nu+1}\cdot \psi)$ is closely related to the image of the negative Dirac operator \eqref{neg-Dirac}, which we will exploit to determine a complement.
In the following let $K_{l,\nu}$ be a complement of $d \ker (\slashed{D}_{-})_{\nu+1}$ in $L^2_{l,\nu}(\Lambda^4_{35})$. By Lemma \ref{Injectivity of negative Dirac-operator} we have
\begin{align}
\label{tangent-space-of-slice}
K_{l,\nu}\cong
\begin{cases}
L^2_{l,\nu}(\Lambda^4_{35}) 
&\text{if}\, \nu\in(-4,-1],
\\
L^2_{l,\nu}(\Lambda^4_{35})/d \ker (\slashed{D}_{-})_{\nu+1}
&\text{if}\, \nu\in(-1,0).
\end{cases}
\end{align} 

\begin{proposition}
\label{prop-decomp-admissable-forms}
We have the decomposition
\begin{align*}
L^2_{l,\nu}(\Lambda^4_{1} \oplus \Lambda^4_7 \oplus \Lambda^4_{35})
=
T_{\psi}(\mathcal{D}_{l+1,\nu+1} \cdot \psi)
\oplus 
K_{l,\nu}
\end{align*}
\end{proposition}
\begin{proof}
Let $\alpha \in L^2_{l,\nu}(\Lambda^4_{1} \oplus \Lambda^4_7 \oplus \Lambda^4_{35})$.
If $\nu\in(-4,0)$, the negative Dirac-operator 
\begin{align*}
(\slashed{D}_{-})_{l+1,\nu+1}\colon
L^2_{l+1,\nu+1}(\Lambda^3_{8})
\rightarrow
L^2_{l,\nu}(\Lambda^4_{1} \oplus \Lambda^4_7)
\end{align*}
is surjective by Corollary \ref{Injectivity of negative Dirac-operator}.
Therefore, we can write 
\begin{align*}
\pi_{1+7}\alpha
=
\slashed{D}_{-}(X\lrcorner \psi) 
=
\pi_{1+7}d(X \lrcorner \psi) 
\end{align*}
for some $ X\lrcorner \psi \in L^2_{l+1,\nu+1}(\Lambda^3_8)$.
Then \begin{align*}
\alpha - d(X\lrcorner \psi) \in L^2_{l,\nu}(\Lambda^4_{35}).
\end{align*}
This proves
\begin{align*}
L^2_{l,\nu}(\Lambda^4_{1} \oplus \Lambda^4_7 \oplus \Lambda^4_{35})
=
T_{\psi}(\mathcal{D}_{l+1,\nu+1} \cdot \psi)
+
L^2_{l,\nu}(\Lambda^4_{35})
\end{align*}
for $\nu\in(-4,0)$. Assume that 
$d(X\lrcorner \psi) \in
T_{\psi}(\mathcal{D}_{l+1,\nu+1} \cdot \psi)
\cap
L^2_{l,\nu}(\Lambda^4_{35})$
for some
$X\lrcorner \psi \in L^2_{l+1,\nu+1}(\Lambda^3_8)$.
Then $\slashed{D}_{-}(X\lrcorner \psi)
=\pi_{1+7}d(X\lrcorner\psi)=0$.
This proves the statement.
\end{proof}

By Proposition \ref{prop-decomp-admissable-forms} we want to determine a slice $\mathcal{S}_{\psi}$ around $\psi$ which  satisfies the properties (S.1)-(S.3) from Theorem \ref{good-slices} and at $\psi$ has the tangent space $K_{l,\nu}$.
A candidate for such a slice is the graph of the map $\Theta$, which we have defined in \eqref{Pi-Theta}, over a neighbourhood of $\psi$ in the affine space $\psi+K_{l,\nu}$. We set
\begin{align*}
\mathcal{S}_{\psi}:= \{\Pi(\eta) = \psi + \eta -\Theta(\eta)\, |\, \eta \in U \subset K_{l,\nu} \},
\end{align*}
where $U$ is a sufficiently small neighbourhood of the origin in $K_{l,\nu}$. 
Properties (S.1) and (S.2) follow from the properties (i)-(iii) of the maps $\Pi$ and $\Theta$. Next we proof property (S.3):

\begin{proposition}
All elements in a neighbourhood of $\psi$ in $\mathcal{R}_{\psi}$
are smooth.
\end{proposition}
\begin{proof}
By the definition of $\mathcal{S}_{\psi}$ each $\hat{\psi}\in\mathcal{R}_{\psi}$ sufficiently close to $\psi$ can be written as
\begin{align*}
\Pi(\xi)
=
\psi+\xi-\Theta(\xi)
\end{align*}
for some $\xi\in L^2_{l,\nu}(\Lambda^4_{35})$. 
Using the Hodge star operator with respect to $\psi$, the condition $d\Pi(\xi)=0$ leads to the equation
\begin{align*}
(d+d^*) \xi - d\Theta(\xi)-*d\Theta(\xi)=0, 
\end{align*}
which can be re-written as
\begin{align*}
(d+d^*)\xi + Q(\xi, \nabla\xi) = R(\xi),
\end{align*}
where $Q(x,y)$ and $R(x)$ are smooth maps which depend on $\psi$ and its derivatives, and 
$Q(x,y)$ is linear in $y$ and $Q(0,y)=0$ for all $y$.
Because the $\mathcal{C}^0_{\nu}$-norm of $\xi$ is controlled by its $L^2_{l,\nu}$-norm via the Sobolev embedding $L^2_{l,\nu}\hookrightarrow \mathcal{C}^{1,\alpha}_{\nu}$, the linear operator 
$L=d+d^*+Q(\xi, \nabla \cdot)$ is $\mathcal{C}^0$-close to $d+d^*$
if the slice $\mathcal{S}_{\psi}$ is chosen sufficiently small and therefore elliptic. Again by Sobolev embedding we can assume the induction hypothesis $\xi\in\mathcal{C}^{q,\alpha}_{\nu}$ for $q\geq 1$. 
Theorem \ref{general-ell-reg} implies $\xi\in\mathcal{C}^{q+1,\alpha}_{\nu}$. By induction we see that $\xi$ is smooth.
We conclude that $\Pi(\xi)$ is smooth because $\Pi$ and $\xi$  are smooth.
\end{proof}

We have now proved that $\mathcal{M}_{\nu}$ is an orbifold under suitable assumptions on the rate $\nu$. To compute its dimension, we determine the tangent space of the pre-moduli space $\mathcal{R}_{\psi}$ at $\psi$. 

\begin{lemma}
\label{lemma-tangent-space-R-psi}
The tangent space of the pre-moduli space $\mathcal{R}_{\psi}$ at $\psi$ is given by
\begin{align*}
T_{\psi}\mathcal{R}_{\psi}
\cong
\begin{cases}
(\mathcal{H}^4_{35})_{\nu}
&\text{if}\, \nu\in(-4,-1],
\\
(\mathcal{H}^4_{35})_{\nu}/d \ker (\slashed{D}_{-})_{\nu+1}
&\text{if}\, \nu\in(-1,0).
\end{cases}
\end{align*}
\end{lemma}
\begin{proof}
By construction we have
\begin{align*}
T_{\psi}\mathcal{R}_{\psi}
=
T_{\psi}\mathcal{X}_{l,\nu} \cap T_{\psi}\mathcal{S}_{\psi}.
\end{align*}
Therefore, by Lemma \ref{tangent-space-torsion-free}
$T_{\psi}\mathcal{R}_{\psi}$ is the kernel of the linear map
\begin{align*}
d\colon
T_{\psi}\mathcal{S}_{l,\nu}
\rightarrow 
d\Omega^4_{l,\nu}.
\end{align*}
The statement follows with formula \eqref{tangent-space-of-slice}.
\end{proof}

\begin{remark}
\label{remark-scaling}
As explained in the beginning of this section the gauge fixing normalises the scale of the AC Spin(7)-structures. However, scaling is still seen by the moduli space $\mathcal{M}_{\nu}$.
In the following we explain that an AC Spin(7)-manifold with decay rate $\nu$ always induces a canonical Spin(7)-deformation via scaling.
Therefore, if there is a torsion-free AC Spin(7)-structure on $M$ which decays to the cone $C$ precisely at rate $\nu\in (-4,0)$,
by Lemma \ref{lemma-tangent-space-R-psi}
the dimension of the space $(\mathcal{H}^4_{35})_{\lambda}$ must increase as $\lambda$ crosses $\nu$. 
This gives a criterion to exclude the existence of torsion-free AC Spin(7)-structures at certain rates.  
 
The vector field $V = r \partial_r$ generates the flow
\begin{align*}
\Phi_{\lambda}(r,x)=(e^{\lambda}r,x)
\end{align*}
on the cone $C(\Sigma)$. The action of $\Phi_{\lambda}$ scales the conical Spin(7)-structure:
\begin{align*}
\Phi_{\log \lambda}^*\psi_{\ig C} = \lambda^4 \psi_{\ig C}.
\end{align*}
We can transplant the vector field to $M$ by setting $\hat{V}=\chi V$. Denote its flow by $\hat{\Phi}_{\lambda}$.
As noted above the rescaled AC Spin(7)-structures $\lambda^4 \psi$
are not in $\mathcal{A}_{\lambda}$ and therefore do not contribute to $\mathcal{M}_{\nu}$.
However, up to asymptotic decay we can reverse scaling by $\lambda^4$ by the action of $\hat{\Phi}_{1/\log \lambda}$.
Set
\begin{align*}
\psi_{\lambda}
:=
\lambda^4 \hat{\Phi}_{(1/\log \lambda)}^*\psi.
\end{align*}
To see that $\psi_{\lambda}$ decays to $\psi$ with rate $\nu$ write
\begin{align*}
\psi(r,x) = dr\wedge \varphi(r,x) + *\varphi(r,x).
\end{align*}
Then we have
\begin{align*}
\psi_{\lambda}(r,x)
=
\lambda^3 dr\wedge \varphi(r/\lambda,x) + \lambda^4 *\varphi(r/\lambda,x). 
\end{align*}
With respect to the norm given by $g_{\ig C}$ we have 
\begin{align*}
&|\psi_{\lambda}(r,x) - \psi_{\ig C}(r,x)|^2
=
|\psi_{\lambda}(r,x) - \lambda^4 \psi_{\ig C}(r/\lambda,x)|^2
\\
=
&
\lambda^3 
|\varphi(r/\lambda,x)-(r/\lambda)^3 \varphi_{\ig \Sigma}(x)|^2
+
\lambda^4 
|*_{\ig \varphi}\varphi(r/\lambda,x)-(r/\lambda)^4 *_{\ig \Sigma}\varphi_{\ig \Sigma}(x)|^2
\\
=
&
\mathcal{O}(r^{\nu})
\end{align*}
because $\psi$ decays to $\psi_{\ig C}$ with rate $\nu$.

The family $\psi_{\lambda}$
induces the infinitesimal deformation
\begin{align*}
\left.
\frac{d}{d\lambda}
\right|_{\lambda=0} \psi_{\lambda}=
4\psi-\mathcal{L}_{\hat{V}}\psi.
\end{align*}
\end{remark}

The results in this section prove

\begin{proposition}
\label{prop-smoothness-of-moduli}
Suppose $\nu \in (-4,0)$ and that $\nu$ and $\nu+1$ are non-critical rates of the operator $d+d^*$. Then the moduli space 
$\mathcal{M}_{\nu}$ is an orbifold
and the orbifold chart $\mathcal{R}_{\psi}$ at $\psi$ has the tangent space 
\begin{align}
\label{tangent-space}
T_{\psi}\mathcal{R}_{\psi}
\cong
\begin{cases}
(\mathcal{H}^4_{35})_{\nu}
&\text{if}\, \nu\in(-4,-1],
\\
(\mathcal{H}^4_{35})_{\nu}/d \ker (\slashed{D}_{-})_{\nu+1}
&\text{if}\, \nu\in(-1,0).
\end{cases}
\end{align}
Furthermore, if the stabiliser $\mathcal{I}_{\psi}$ of $\psi$ is trivial or acts trivially on the orbifold chart $\mathcal{R}_{\psi}$, the moduli space $\mathcal{M}_{\nu}$ is smooth in a neighbourhood of $\psi$. In particular, this is true (after possibly shrinking $\mathcal{R}_{\psi}$) if the projection $T_{\psi}\mathcal{R}_{\psi}\rightarrow H^4(M,\mathbb{R})$ is injective.
\end{proposition}

\section{Computation of infinitesimal deformations}

\label{comp-inf-defo}

The aim of this section is to give a more precise description of the infinitesimal deformations \eqref{tangent-space}. With the terminology introduced in Definition \ref{def-K} and section \ref{section-topology} we will show

\begin{proposition}
\label{first computation of tangent space}
Let $(M,\psi,g)$ be an AC Spin(7)-manifold.
If $\nu > -4$, we have
\begin{align*}
(\mathcal{H}^4_{35})_{\nu}
\cong
(\mathcal{H}^4_{-})_{L^2}
\oplus
\im\Upsilon^4
\oplus
\bigoplus_{\lambda \in \mathcal{D}(d_{\mathrm{ASD}})
\cap (-4,\nu)}
\mathcal{K}_{ \mathrm{ASD}}(\lambda)
.
\end{align*}
\end{proposition}
Proposition \ref{first computation of tangent space}
will follow from Proposition \ref{ASD kernel change at rates > -4} and 
Corollary \ref{ASD change at -4}. 
We need to study how $(\mathcal{H}^4_{35})_{\nu}$ changes
as $\nu$ passes a critical rate. Because $(\mathcal{H}^4_{35})_{\nu}=\ker (d_{\mathrm{ASD}})_{l,\nu}$ and $(d_{\mathrm{ASD}})_{l,\nu}$ is surjective for generic rates $\nu > -4$ by Lemma \ref{surjectivity of main operator}, this corresponds to the change in $\ind\ (d_{\mathrm{ASD}})_{l,\nu}$ as $\nu$ crosses a critical rate. In the introduction we have explained the index change at critical rates for uniformly elliptic operators. However, $d_{\mathrm{ASD}}$ is not elliptic. Therefore, we need to adapt Theorem \ref{index change for elliptic operators} to our non-elliptic setting. To simplify the presentation we will first explain how it can be adjusted to the non-elliptic operator $d+d^*|_{\Omega^k}$. The main ingredient in the proof of Theorem \ref{index change for elliptic operators} is Theorem \ref{Lockhart-McOwen-key}. We will first adapt this to our situation. Compare with \cite[Lemma 4.28]{KL}.

\begin{proposition}\label{faster than expected decay for k-forms}
Let $(M,\psi)$ be an AC Spin(7)-manifold of rate $\nu$.
Let $\lambda_0$ be a critical rate for $d+d^*$
and let $\beta_2 < \beta_1$ be two non-critical rates for 
$d+d^*$ such that $\lambda_0$ is the unique critical rate for the operator $d+d^*$ in the interval $[\beta_2,\beta_1]$
and $\lambda_0 + \nu < \beta_2$. 

If $\gamma \in \Omega^k_{l+1,\beta_1}$
with $(d+d^*_{\ig M}) \gamma \in 
\Omega^{\bullet}_{l,\beta_2-1}$, then there exist
unique
$\eta \in \mathcal{K}_{\Lambda^k}(\lambda_0)$
and $\tilde{\gamma} \in \Omega^k_{l+1,\beta_2}$
with
\begin{align}
\gamma = \chi \eta + \tilde{\gamma}.
\end{align}
Moreover, $\eta$ and $\tilde{\gamma}$ depend 
linearly on $\gamma$. Here $\chi$ is the cut-off function from Definition \ref{def-radial-function}.
\end{proposition}
\begin{proof}
We cannot apply Theorem \ref{Lockhart-McOwen-key} to $d+d^*_{\ig M}|_{\Omega^k}$, but we can embed $\Omega^k \subset \Omega^{\bullet}$ and then apply Theorem \ref{Lockhart-McOwen-key} to
$d+d^*_{\ig M}: \Omega^{\bullet}\rightarrow \Omega^{\bullet}.$
More specifically,  
there exist
$\omega \in \mathcal{K}_{d+d^*}(\lambda_0)$
and $\tilde{\gamma}\in \Omega_{l+1,\beta_2}^{\bullet}$
such that 
\begin{align*}
\gamma = \omega + \tilde{\gamma}
\end{align*}
on the end.
The price we have to pay for using Theorem \ref{Lockhart-McOwen-key} is that  a priori $\omega$ can be a mixed degree form. 
Therefore, we need to show that all except the degree $k$ part of $\omega$ vanish. 
Because $\gamma$ is a $k$-form, for $ l \neq k$
the $l$-form component of $\omega$
has to decay with rate $\beta_2$ to cancel with the degree $l$ component of $\tilde{\gamma}$. However, each non-zero degree component of 
$\omega$ is homogeneous of rate $\lambda_0 > \beta_2$. Therefore, $\omega$ is a pure degree $k$-form.
Finally, it is straightforward that $\tilde{\gamma}$ is purely of degree $k$ as well.
\end{proof}

\begin{proposition}\label{faster than expected decay ASD forms}
Let $(M,\psi)$ be an AC Spin(7)-manifold of rate $\nu$.
Let $\lambda_0$ be a critical rate for $d_{\mathrm{ASD}}$
and let $\beta_2 < \beta_1$ be two non-critical rates for 
$d+d^*$ such that $\lambda_0$ is the unique critical rate for the operator $d+d^*$ in the interval $[\beta_2,\beta_1]$
and $\lambda_0 + \nu < \beta_2$. 

If $\gamma \in L^2_{l+1,\beta_1}(\Lambda^4_{35})$
with $d \gamma \in 
\Omega^5_{l,\beta_2-1}$, then 
there exist unique $\omega\in\mathcal{K}_{\mathrm{ASD}}(\lambda_0)$
and $\hat{\gamma}\in L^2_{l+1,\beta_2}(\Lambda^4)$ such that
\begin{align*}
\gamma = \chi\omega + \hat{\gamma}.
\end{align*}
This decomposition depends linearly on $\gamma$.
\end{proposition}
\begin{proof}
Because $\gamma$ is anti-self-dual, we have $d^*_{\ig M} \gamma = -*_{\ig M} d *_{\ig M} \gamma = *_{\ig M} d \gamma$.
Because the Hodge-star is an isometry and 
$d \gamma \in 
\Omega^5_{l,\beta_2-1}$, we know that
$d^*_{\ig M} \gamma \in 
\Omega^3_{l,\beta_2-1}$.
By embedding $\Lambda^4_{35} \subset \Lambda^4$ we can use
Proposition \ref{faster than expected decay for k-forms}
to get $\omega\in \mathcal{K}_{\Lambda^4}(\lambda_0)$ and $\hat{\gamma}\in \Omega^4_{l+1,\beta_2}$ such that
\begin{align*}
\gamma = \chi \omega + \hat{\gamma}.
\end{align*}
Projecting on the self-dual part gives 
\begin{align}
0 
&=
\chi(\omega + *_{\ig C} \omega)
+
\chi( *_{\ig M} - *_{\ig C}) \omega
+ \hat{\gamma}
+ *_{\ig M} \hat{\gamma}.
\label{asd-projection1}
\end{align}
By Lemma \ref{additional-decay} the middle term in \eqref{asd-projection1} decays like $\lambda_0+\nu$. Because $\lambda_0+\nu < \beta_2$, all
terms on the right-hand side of \eqref{asd-projection1} except $(\omega + *_{\ig C} \omega)$ decay with rate $\beta_2 < \lambda_0$ while $(\omega + *_{\ig C} \omega)$ decays with rate $\lambda_0$. Therefore, $(\omega + *_C \omega)$ has to vanish, i.e. $\omega$
is of type 35 with respect to the Spin(7)-structure on the cone, and in particular
$\omega \in \mathcal{K}_{\mathrm{ASD}}(\lambda_0)$.

\end{proof}

\begin{proposition}
\label{existence of crit rate cross map}
Let $\lambda_0$ be a critical rate for the operator $d+d^*_{\ig M}|_{\Omega^k}$, and choose $\varepsilon > 0$ small enough such that $\lambda_0$ is the unique critical rate for the operator $d+d^*_{\ig M}|_{\Omega^k}$ in the interval $(\lambda_0-\varepsilon,\lambda_0+\varepsilon)$ and $\lambda_0 + \nu < \lambda_0 -\varepsilon$. Then there exists an injective linear map 
\begin{align*}
T^k_{\lambda_0}\colon 
\mathcal{H}^k_{\lambda_0+\varepsilon}
/
\mathcal{H}^k_{\lambda_0-\varepsilon}
\rightarrow
\mathcal{K}_{\Lambda^k}(\lambda_0).
\end{align*} 
In the same setting for the operator $d_{\mathrm{ASD}}$ there exists
an injective linear map
\begin{align*}
T^{\mathrm{ASD}}_{\lambda_0}\colon 
(\mathcal{H}^4_{35})_{\lambda_0+\varepsilon}
/
(\mathcal{H}^4_{35})_{\lambda_0-\varepsilon}
\rightarrow
\mathcal{K}_{\mathrm{ASD}}(\lambda_0).
\end{align*} 
\end{proposition}
\begin{proof}
Let $\gamma\in\mathcal{H}^k_{\lambda_0+\varepsilon}$. 
Because $(d+d^*_{\ig M})\gamma = 0$, we are in the situation of 
Proposition \ref{faster than expected decay for k-forms}.
Hence there is a unique $\eta\in\mathcal{K}_{\Lambda^k}(\lambda_0)$ 
such that 
\begin{align*}
\gamma = \chi \eta + \mathcal{O}(r^{\lambda_0-\varepsilon}).
\end{align*}
We set $T^k_{\lambda_0}(\gamma) = \eta$.
It is clear that $\eta = 0$ if $\gamma\in\mathcal{H}^k_{\lambda_0-\varepsilon}$. Hence $T^k_{\lambda_0}$ is well-defined. Because $\eta$ depends linearly on $\gamma$, the map is linear. If $T^k_{\lambda_0}(\gamma) = 0$, then $\gamma = \mathcal{O}(r^{\lambda_0-\varepsilon})$. Therefore, $T^k_{\lambda_0}$ is injective. 
The statement for $d_{\mathrm{ASD}}$ follows analogously by using Proposition \ref{faster than expected decay ASD forms}.
\end{proof}

\begin{proposition}
\label{ASD kernel change at rates > -4}
If $\lambda_0 > -4$, 
the map $T^{\mathrm{ASD}}_{\lambda_0}$ is an isomorphism.
\end{proposition}
\begin{proof}
We need to prove the surjectivity of $T^{\mathrm{ASD}}_{\lambda_0}$. Suppose $\eta\in\mathcal{K}_{\mathrm{ASD}}(\lambda_0)$. As $\lambda_0 \neq -4$,   
$\eta$ is exact on the cone by Remark \ref{remark-char-asd}, i.e. there exists a 3-form $\xi$ on the cone such that $\eta=d\xi$. But then $d(\chi \eta) = d(\chi \eta - d(\chi \xi))\in d \Omega^4_{\mathrm{cs}}$. $(d_{\mathrm{ASD}})_{l,\lambda-\epsilon}$ is surjective 
by Lemma \ref{surjectivity of main operator}
as $\varepsilon$ can be chosen small enough such that $\lambda-\varepsilon > -4$. Therefore, there exists $\hat{\gamma}\in L^2_{l,\lambda-\varepsilon}(\Lambda^4_{35})$ such that $\gamma = \chi \eta + \hat{\gamma}$ is closed.
\end{proof}

\subsection{The exceptional rate $-4$}

To finish the proof of Proposition \ref{first computation of tangent space}
it remains to compute the index change for 
the operator $(d_{\mathrm{ASD}})_{l,\lambda}$
at the exceptional critical rate $\lambda = -4$.
We are going to prove a more general statement by considering the operator $d+d^*_{\ig M}$ restricted to 4-forms at the critical rate $\lambda=-4$. The main result of this section is

\begin{proposition}
\label{exceptional index change for 4 forms}
The map $T^4_{-4}$ takes values in $r^{-1} dr\wedge(\im \Upsilon^3)^{\perp}+ \im \Upsilon^4$,
and 
\begin{align*}
T^4_{-4}
\colon
\mathcal{H}^{\mathrm{4}}_{-4+\varepsilon}
/
\mathcal{H}^{\mathrm{4}}_{-4-\varepsilon}
\rightarrow
r^{-1} dr\wedge(\im \Upsilon^3)^{\perp}+ \im \Upsilon^4
\end{align*}
is an isomorphism.
\end{proposition}

The corresponding statement for the operator $d_{\mathrm{ASD}}$ is a simple consequence of Proposition \ref{exceptional index change for 4 forms}.

\begin{corollary}
\label{ASD change at -4}
The map $T^{\mathrm{ASD}}_{-4}$ takes values in the space
\begin{align*}
\{r^{-1} dr\wedge(-*_{\ig\Sigma}\beta)+\beta\ |\ \beta\in \im\Upsilon^4 \} \cong \im \Upsilon^4,
\end{align*}
and
\begin{align*}
T^{\mathrm{ASD}}_{-4}\colon 
(\mathcal{H}^4_{35})_{-4+\varepsilon}
/
(\mathcal{H}^4_{35})_{-4-\varepsilon}
\rightarrow
\mathrm{im} \Upsilon^4
\end{align*}
is an isomorphism.
\end{corollary}
\begin{proof}
Because 
\begin{align*}
(\mathcal{H}^4_{35})_{-4+\varepsilon}
/
(\mathcal{H}^4_{35})_{-4-\varepsilon}
\subset
\mathcal{H}^{\mathrm{4}}_{-4+\varepsilon}
/
\mathcal{H}^{\mathrm{4}}_{-4-\varepsilon},
\end{align*}
the map $T^{\mathrm{ASD}}_{-4}$ is the map $T^4_{-4}$ restricted to anti-self-dual forms. The statement follows from
Theorem \ref{exceptional index change for 4 forms} 
and the description of $\mathcal{K}_{\mathrm{ASD}}(-4)$ in Remark \ref{remark-char-asd}.
\end{proof}

In the remainder of this subsection we will prove Proposition \ref{exceptional index change for 4 forms} in several steps.

\begin{lemma}
\label{image of T contained in image of Upsilon}
The map $T^{4}_{-4}$ takes values in 
$r^{-1} dr\wedge(\im \Upsilon^3)^{\perp}+\im \Upsilon^4$,
i.e. there is an injective linear map
\begin{align*}
T^{4}_{-4}
\colon
\mathcal{H}^4_{-4+\varepsilon}
/
\mathcal{H}^4_{-4-\varepsilon}
\rightarrow
r^{-1} dr\wedge (\im \Upsilon^3)^{\perp}+\im \Upsilon^4
.
\end{align*}
\end{lemma}
\begin{proof}
Let $\gamma\in\mathcal{H}^4_{-4+\varepsilon}$. Then by 
Proposition \ref{existence of crit rate cross map}
and Corollary \ref{forms on cone at rate -4}
there exist harmonic forms $\alpha\in\Omega^3(\Sigma)$ and $\beta\in\Omega^4(\Sigma)$ such that
\begin{align*}
\gamma
=
\chi (r^{-1} dr\wedge \alpha + \beta) 
+
\mathcal{O}(r^{-4-\varepsilon})
\end{align*}
and $T^4_{-4}(\gamma) = r^{-1} dr\wedge \alpha + \beta$.
By Lemma \ref{fast decay -> exact on end}
the part of $\gamma$ which decays like $-4-\varepsilon$ is exact on the end. Therefore, $\Upsilon^4([\gamma]) = [\beta]$ and $[\beta]\in \im \Upsilon^4$.
The same argument for $*\gamma$ gives $[*\alpha]\in \mathrm{im}\Upsilon^4$ and hence $[\alpha]\in (\mathrm{im}\Upsilon^3)^{\perp}$ by Lemma \ref{orthogonality of image}.
\end{proof}

The main difficulty in proving Proposition \ref{exceptional index change for 4 forms} is to show surjectivity.
The next Lemma is a first step towards this goal. However, because of the structure of $\mathcal{K}_{\Lambda^4}(-4)$ more work will be needed later on.

\begin{lemma}
\label{half sufficiency}
\begin{compactenum}[(i)]
\item
Let $\beta\in\im \Upsilon^4$. Then there exist $\alpha\in (\im \Upsilon^3)^{\perp}$ and $\gamma\in\mathcal{H}^4_{-4+\varepsilon}$ such that 
$T^4_{-4}(\gamma) = r^{-1} dr\wedge\alpha+\beta$.
\item
Let $\alpha\in(\im \Upsilon^3)^{\perp}$. Then there exist $\beta\in \im \Upsilon^4$ and $\gamma\in\mathcal{H}^4_{-4+\varepsilon}$ such that 
$T^4_{-4}(\gamma) = r^{-1} dr\wedge\alpha+\beta$.
\end{compactenum}
\end{lemma}
\begin{proof}
(i)
By Lemma \ref{lift} there exists $\zeta \in \Omega^4_{\mathrm{cs}}(M)$
such that $\xi = \chi \beta + \zeta$ is closed. Because $\xi \in \Omega^4_{l, -4+\varepsilon}$ we can apply the Hodge decomposition from Proposition \ref{Hodge-decomposition}
to deduce that there is a $\gamma\in \mathcal{H}^4_{-4+\varepsilon}$ cohomologous to $\xi$. 
In particular $\Upsilon^4([\gamma])=\Upsilon^4([\xi]) = \beta$.
By Lemma \ref{image of T contained in image of Upsilon} there is some $\alpha\in \im (\Upsilon^3)^{\perp}$ 
with $\gamma = \chi (r^{-1} dr\wedge\alpha+\beta)+\mathcal{O}(r^{-4-\varepsilon})$ and $T^4_{-4}(\gamma)=r^{-1} dr\wedge\alpha+\beta$.
\\
(ii) 
By Lemma \ref{orthogonality of image}  $*\alpha\in \im \Upsilon^4$. The statement follows from (i).
\end{proof}
\begin{remark}
\label{description of problem}
Note that in Lemma \ref{half sufficiency} (i) we cannot choose $\alpha$. It just says that there exists some. 
This in particular means that we cannot yet prove that 
$T^{4}_{-4}
\colon
\mathcal{H}^4_{-4+\varepsilon}
/
\mathcal{H}^4_{-4-\varepsilon}
\rightarrow
r^{-1} dr\wedge (\im \Upsilon^3)^{\perp}+\im \Upsilon^4$
is surjective, and therefore an isomorphism.
The reason is that in the proof of Lemma \ref{half sufficiency}
we pass from a closed form $\xi = \chi \beta + \mathcal{O}(r^{-4-\varepsilon})$ to a cohomologous closed and co-closed form $\gamma = \xi + d\zeta$ for some $\zeta\in \Omega^3_{l+1,-3+\varepsilon}$. Note that $d(\chi \log r \alpha)= \chi (r^{-1} dr\wedge\alpha) + \mathcal{O}(r^{-4-\varepsilon})$
and $\chi \log r \alpha\in \Omega^3_{l+1,-3+\varepsilon}$.
Hence by transitioning from $\xi$ to $\gamma$ a priori there could be introduced a 3-form $\alpha$ such that $\gamma = \chi(r^{-1} dr\wedge\alpha + \beta)+\mathcal{O}(r^{-4-\varepsilon})$. It will take significantly more effort to rule this out. 
This difficulty is unique to change in harmonic middle dimensional degree forms at the $L^2$-rate of even dimensional AC manifolds. On odd dimensional cones such as in the $\mathrm{G}_2$-setting this difficulty does not appear.
\end{remark}

Our idea to overcome this difficulty is to interpret $\mathcal{H}^4_{-4+\varepsilon}/\mathcal{H}^4_{-4-\varepsilon}$ as the kernel change of the elliptic operator 
\begin{align*}
d+d^*\colon \Omega^{\mathrm{even}}_{l,-4 \pm \varepsilon}
\rightarrow
\Omega^{\mathrm{odd}}_{l-1,-5 \pm \varepsilon}.
\end{align*}
This will allow us to instead compute the change in kernel at the critical rate $\lambda = -3$ of the adjoint operator
\begin{align*}
d+d^*\colon \Omega^{\mathrm{odd}}_{l,-3 \mp \varepsilon}
\rightarrow
\Omega^{\mathrm{even}}_{l-1,-4 \mp \varepsilon},
\end{align*}
which is easier. In the following we make this idea precise.

\begin{lemma}
\label{4-form change = even form change}
We have
\begin{align*}
\mathcal{H}^4_{-4+\varepsilon}
/
\mathcal{H}^4_{-4-\varepsilon}
\cong
\mathcal{H}^{\mathrm{even}}_{-4+\varepsilon}
/
\mathcal{H}^{\mathrm{even}}_{-4-\varepsilon}.
\end{align*}
\end{lemma}
\begin{proof}
By Lemma \ref{individual degree closed and coclosed}
every degree component of a form in $\mathcal{H}^{\mathrm{even}}_{-4+\varepsilon}$ is closed and co-closed
and therefore
\begin{align*}
\mathcal{H}^{\mathrm{even}}_{-4+\varepsilon}
/
\mathcal{H}^{\mathrm{even}}_{-4-\varepsilon}
\cong
\bigoplus_{k=0}^4 
\mathcal{H}^{2k}_{-4+\varepsilon}
/
\mathcal{H}^{2k}_{-4-\varepsilon}.
\end{align*}
By Corollary \ref{forms on cone at rate -4} we have $\mathcal{K}_{\Lambda^k}(-4) = 0$ for $ k = 0, 2, 6, 8$. Therefore by Proposition \ref{existence of crit rate cross map} we have
$\mathcal{H}^k_{-4+\varepsilon}
/
\mathcal{H}^k_{-4-\varepsilon} = 0$
for $k = 0, 2, 6, 8$. The statement follows. 
\end{proof}

If we consider forms of degrees 3 and 5 at rate $\lambda = -3$, we don't have the problem described in Remark \ref{description of problem}.

\begin{lemma}\label{crossing -3 for deg 3 and 5}
\begin{compactenum}[(i)]
\item
The map $T^3_{-3}$ takes values in $\im \Upsilon^3$
and 
\begin{align*}
T^3_{-3}\colon
\mathcal{H}^3_{-3+\varepsilon}
/
\mathcal{H}^3_{-3-\varepsilon}
\rightarrow
\im \Upsilon^3
\end{align*}
is an isomorphism.
\item
The map $T^5_{-3}$ takes values in $r dr\wedge(\im \Upsilon^4)^{\perp}$
and 
\begin{align*}
T^5_{-3}
\colon
\mathcal{H}^5_{-3+\varepsilon}/
\mathcal{H}^5_{-3-\varepsilon}
\rightarrow
r dr\wedge(\im \Upsilon^4)^{\perp}
\end{align*}
is an isomorphism.
\end{compactenum}
\end{lemma}
\begin{proof}
(i)
By Corollary \ref{forms on cone at rate -4} we have
$\mathcal{K}_{\Lambda^3}(-3) = H^3(\Sigma,\mathbb{R})$. The proof that the image of $T^3_{-3}$ is contained in $\im \Upsilon^3$ is analogous to the proof of Lemma \ref{image of T contained in image of Upsilon}.
The proof that it is surjective onto $\im \Upsilon^3$ is analogous to the proof of Lemma \ref{half sufficiency}.
\\
(ii)
The statement follows from (i) by applying the Hodge-$*$ operator.
\end{proof}

We have
\begin{align*}
\mathcal{H}^3_{-3+\varepsilon}
/
\mathcal{H}^3_{-3-\varepsilon}
\oplus
\mathcal{H}^5_{-3+\varepsilon}
/
\mathcal{H}^5_{-3-\varepsilon}
\subset
\mathcal{H}^{\mathrm{odd}}_{-3+\varepsilon}
/
\mathcal{H}^{\mathrm{odd}}_{-3-\varepsilon}.
\end{align*}
To conclude equality we would need that for any element in $\mathcal{H}^{\mathrm{odd}}_{-3+\varepsilon}$ all individual degree components are closed and co-closed. In general this statement starts to fail at rate $-3$ (see Lemma \ref{individual degree closed and coclosed}). However, because each individual degree component of the leading order term in $\mathcal{K}_{\mathrm{odd}}(-3)$
is closed and co-closed, we can improve  Lemma \ref{individual degree closed and coclosed} to rate $-3+\varepsilon$ for odd degree forms.

\begin{lemma}\label{extended partial integration}
For all elements
in $\mathcal{H}^{\mathrm{odd}}_{-3+\varepsilon}$
each individual degree component is closed and co-closed.
\end{lemma}
\begin{proof}
Let $\gamma=\gamma_1 + \gamma_3 + \gamma_5 + \gamma_7\in\mathcal{H}^{\mathrm{odd}}_{-3+\varepsilon}$.
By Theorem \ref{Lockhart-McOwen-key} there exist $\eta\in\mathcal{K}_{\mathrm{odd}}(-3)$ 
and $\hat{\gamma}\in \Omega^{\mathrm{odd}}_{l,-3-\varepsilon}$
such that 
\begin{align*}
\gamma = \chi \eta + \hat{\gamma}.
\end{align*}
Because each degree component of $\eta$ is closed and co-closed on the cone by Lemma \ref{odd-degree at rate -3}, by Lemma \ref{additional-decay}
$d\gamma_k\in\Omega^{k+1}_{l-1,\-4-\varepsilon}$ and $d_{\ig M}^*\gamma_k\in \Omega^{k-1}_{l-1,-4-\varepsilon}$ for $k = 1,3,5,7$. Therefore, we can apply Proposition \ref{partial integration} to integrate by parts:
\begin{align*}
\| d \gamma_k \|^2_{L^2} + \| d^*\gamma_k \|^2_{L^2}
=
\langle d\gamma_k, d\gamma_k \rangle_{L^2}
+
\langle d_M^*\gamma_k , d_M^*\gamma_k \rangle_{L^2}
=
\langle \Delta \gamma_k, \gamma_k \rangle_{L^2}
=
0.
\end{align*}
\end{proof}

\begin{corollary}
\label{kernel change of odd forms}
We have
\begin{align*}
\mathcal{H}^{\mathrm{odd}}_{-3+\varepsilon}
/
\mathcal{H}^{\mathrm{odd}}_{-3-\varepsilon}
=
\mathcal{H}^3_{-3+\varepsilon}
/
\mathcal{H}^3_{-3-\varepsilon}
\oplus
\mathcal{H}^5_{-3+\varepsilon}
/
\mathcal{H}^5_{-3-\varepsilon}
=
\im \Upsilon^3 + r dr\wedge (\im \Upsilon^4)^{\perp}.
\end{align*}
\end{corollary}
\begin{proof}
This follows immediately from Lemmas \ref{crossing -3 for deg 3 and 5} and \ref{extended partial integration}.
\end{proof}

\begin{proof}[Proof of Theorem \ref{exceptional index change for 4 forms}]
By Lemma \ref{4-form change = even form change} we can compute the change in kernel of the operator $(d+d^*_M)^{\mathrm{even}}_{\lambda}$ at rate $\lambda = -4$ instead of the operator $d+d^*_M$ restricted to 4-forms. The kernel change of $(d+d^*)^{\mathrm{even}}_{\lambda}$ corresponds to the cokernel change of the operator $(d+d^*_M)^{\mathrm{odd}}_{\lambda}$ at rate $\lambda = -3$. Using Theorem \ref{index change for elliptic operators} and Corollary \ref{kernel change of odd forms} we have a complete understanding of that. In formulas:
\begin{gather*}
\dim \mathcal{H}^4_{-4+\varepsilon}
-
\dim \mathcal{H}^4_{-4-\varepsilon}
=
\dim \mathcal{H}^{\mathrm{even}}_{-4+\varepsilon}
-
\dim \mathcal{H}^{\mathrm{even}}_{-4-\varepsilon}
\\
=
\dim \ker (d+d^*_{\ig M})^{\mathrm{even}}_{-4+\varepsilon}
-
\dim \ker (d+d^*_{\ig M})^{\mathrm{even}}_{-4-\varepsilon}
\\
=
\dim \coker(d+d^*_{\ig M})^{\mathrm{odd}}_{-3-\varepsilon}
-
\dim \coker (d+d^*_{\ig M})^{\mathrm{odd}}_{-3+\varepsilon}
\\
=
(
\ind (d+d^*_{\ig M})^{\mathrm{odd}}_{-3+\varepsilon}
- \ind (d+d^*_{\ig M})^{\mathrm{odd}}_{-3-\varepsilon}
)
-
(
\dim \ker (d+d^*_{\ig M})^{\mathrm{odd}}_{-3+\varepsilon}
-
\dim \ker (d+d^*_{\ig M})^{\mathrm{odd}}_{-3-\varepsilon}
)
\\
=
(
\dim H^3(\Sigma,\mathbb{R})+\dim H^4(\Sigma,\mathbb{R})
)
-
(
\dim \im \Upsilon^3 + \dim (\Upsilon^4)^{\perp}
)
\\
=
\dim (\im \Upsilon^3)^{\perp}
+
\dim \im \Upsilon^4.
\end{gather*}
Therefore the injection 
\begin{align*}
T^{4}_{-4}
\colon
\mathcal{H}^4_{-4+\varepsilon}
/
\mathcal{H}^4_{-4-\varepsilon}
\rightarrow
r^{-1} dr\wedge (\im \Upsilon^3)^{\perp}+\im \Upsilon^4
.
\end{align*}
from Lemma \ref{image of T contained in image of Upsilon}
is surjective.
\end{proof}

In the following we describe an alternative proof of the surjectivity in Proposition \ref{exceptional index change for 4 forms}.
Let $\beta$ be the harmonic representative of any class in $\im \Upsilon^4$.
To prove Theorem \ref{exceptional index change for 4 forms} it is enough to find  $\xi \in \mathcal{C}^{\infty}_{-4-\varepsilon}(\Lambda^4)$ such that 
\begin{align}
\label{2nd proof-eqn 1}
(d+d^*)(\chi \beta) = (d+d^*)\xi.
\end{align} 
Then $\gamma = \chi \beta-\xi \in \mathcal{H}^4_{-4+\varepsilon}$ and $\Upsilon^4[\gamma] = [\beta]$. By using linearity and the Hodge-$*$ operator we then can solve equation 
(\ref{2nd proof-eqn 1}) if we replace $\beta$ by any element 
in $r^{-1} dr\wedge(\im \Upsilon^3)^{\perp}+\im \Upsilon^4$.

The main ingredient in the alternative proof is
Lemma \ref{orthogonality of image}. It allows to show that the obstructions to solve (\ref{2nd proof-eqn 1}) vanish. Let $m$ be an arbitrary 
integer at least 1.
Even though $\chi\beta$ lies just in $\Omega^4_{m,-4+\varepsilon}$,
we get an improved decay rate for $(d+d^*)(\chi\beta)$.
Because $\beta$ is closed and co-closed on the cone by Lemma \ref{additional-decay} we have $(d+d^*)(\chi \beta)\in \Omega^{\mathrm{odd}}_{m-1,-5-\varepsilon}$ if we choose $\varepsilon$  small enough
such that $2\varepsilon < - \nu$.
To sum up we consider the operator
\begin{align*}
d+d^* \colon
\Omega^{\mathrm{even}}_{m,-4-\varepsilon}\rightarrow
\Omega^{\mathrm{odd}}_{m-1,-5-\varepsilon}
\end{align*}
and want to determine if $(d+d^*)(\chi \beta)$
is in the image.
The adjoint operator is given by
\begin{align*}
d+d^* \colon
\Omega^{\mathrm{odd}}_{l,-3+\varepsilon}\rightarrow
\Omega^{\mathrm{even}}_{l-1,-4+\varepsilon}.
\end{align*}
By the Fredholm alternative we need to show
\begin{align*}
\langle
(d+d^*)(\chi \beta), \sigma
\rangle_{L^2}
=0
\end{align*}
for any $\sigma\in \mathcal{H}^{\mathrm{odd}}_{-3+\varepsilon}$. 
By Proposition \ref{existence of crit rate cross map} and Corollary \ref{forms on cone at rate -4}
we can write $\sigma = \sigma_+ + \sigma_-$,
where 
\begin{align*}
\sigma_+ =
\chi (\zeta + r dr\wedge\eta)
\end{align*}
with $\zeta\in \im \Upsilon^3$, $\eta\in (\im \Upsilon^4)^{\perp}$,
and $\sigma_{-}\in\Omega^{\mathrm{odd}}_{l,-3-\varepsilon}$
.
Standard integration by parts from Lemma \ref{partial integration}
gives 
\begin{align*}
\langle
(d+d^*)(\chi \beta), \sigma
\rangle
=
\langle
(d+d^*)(\chi \beta), \sigma_+
\rangle
+
\langle
(d+d^*)(\chi \beta), \sigma_-
\rangle
\\
=
\langle
(d+d^*)(\chi \beta), \sigma_+
\rangle
+
\langle
\chi \beta, (d+d^*)\sigma_-
\rangle
\\
=
\langle
(d+d^*)(\chi \beta), \sigma_+
\rangle
-
\langle
\chi \beta, (d+d^*)\sigma_+
\rangle
.
\end{align*}
Therefore it is enough to prove the integration by parts
\begin{align*}
\langle
(d+d^*)(\chi \beta), \sigma_+
\rangle
=
\langle
\chi \beta, (d+d^*)\sigma_+
\rangle.
\end{align*}
Note that this does not merely follow 
from Lemma \ref{partial integration}
because the smallest possible $\lambda$
such that $\chi \beta \in L^2_{k,\lambda}(\Lambda^{\mathrm{even}})$
is strictly greater than $-4$
and the smallest possible $\mu$ such that 
$\sigma_+ \in L^2_{k,\mu}(\Lambda^{\mathrm{odd}})$
is strictly greater than $-3$, and hence the sum is 
strictly greater than $-7$.
Therefore, we are in a situation in which Lemma
\ref{partial integration} fails. 
The idea is to use Lemma \ref{orthogonality of image}
to adapt the proof of integration by parts to this situation.
First note that we can write
$*_{\ig M}(\chi r dr\wedge\eta)=\chi *_{\ig \Sigma}\eta +\omega$
with some $\omega\in\mathcal{C}^{\infty}_{-3+\nu}(\Lambda^3)$.
We have
\begin{align*}
&
\langle
d(\chi \beta), \sigma_+
\rangle
=
\langle
d(\chi \beta), \chi r dr\wedge\eta
\rangle
=
\int_M
(d(\chi\beta))
\wedge
*
(\chi r dr\wedge\eta)
\\
=&
\int_M
d
\big(
(\chi \beta)
\wedge
*
(\chi r dr\wedge\eta)
\big)
-
\int_M
(\chi \beta)
\wedge
d*
(\chi r dr\wedge\eta)
\\
=&
\int_M
d
\big(
(\chi \beta)
\wedge
(\chi *\eta+\omega)
\big)
+
\int_M
(\chi \beta)
\wedge
*d^*
(\chi r dr\wedge\eta)
\\
=&
\lim_{r\rightarrow \infty}
\int_{\{r\}\times \Sigma}
\beta\wedge *\eta
+
\lim_{r\rightarrow \infty}
\int_{\{r\}\times \Sigma}
\beta\wedge 
\omega|_{\{r\}\times \Sigma}
+
\langle
\chi \beta, d^*\sigma_+
\rangle
\\
=&
\lim_{r\rightarrow \infty}
r^7 
\langle
\beta,\eta
\rangle_{L^2(\Sigma)}
+
\lim_{r\rightarrow \infty}
\int_{\{r\}\times \Sigma}
\beta\wedge 
\omega|_{\{r\}\times \Sigma}
+
\langle
\chi \beta, d^*\sigma_+
\rangle
\\
=&
\langle
\chi \beta, d^*\sigma_+
\rangle.
\end{align*}
In the second to last line
the first limit vanishes because 
$\beta \perp_{L^2} \eta$
and the second limit vanishes because
the integrand decays with rate $-4-3 + \nu < -7$.

$\langle
d^*(\chi \beta), \sigma_+
\rangle
=
\langle
\chi\beta, d\sigma_+
\rangle
$
follows similarly.

\subsection{Summary}

We now summarise our results and give a precise formulation of the main theorem stated in the introduction. In Proposition \ref{prop-smoothness-of-moduli}
we have seen that for generic rates in the non-$L^2$ regime the moduli space $\mathcal{M}_{\nu}$ is an orbifold and that infinitesimal deformations of a torsion-free AC Spin(7)-structure $\psi$ are related to closed anti-self-dual 4-forms on $(M,\psi)$. So far we have proven Proposition \ref{first computation of tangent space} in this section, showing that the jump of $(\mathcal{H}^4_{35})_{\nu}$ at a critical rate $\lambda$ is given by $\mathcal{K}_{\mathrm{ASD}}(\lambda)$. Before we formulate our main theorem, we relate forms in  $\mathcal{K}_{\mathrm{ASD}}(\lambda)$ to solutions of a differential equation purely on the link $(\Sigma,\varphi_{\ig\Sigma})$ of the asymptotic cone.

\begin{proposition}
\label{prop-reformulation}
Define
\begin{align*}
\mathcal{E}(\Sigma,\varphi_{\ig\Sigma},\lambda)
:=
\{
\zeta\in\Omega^3_{27}(\Sigma)
|\
d\zeta
=
-(\lambda+4) *\zeta
\}.
\end{align*}
Then we have
\begin{align*}
\mathcal{K}_{\mathrm{ASD}}(\lambda)
\cong
\begin{cases}
\mathcal{E}(\Sigma,\varphi_{\ig\Sigma},\lambda)
&
\textrm{if}\, \lambda\in (-4,-1],
\\
\mathcal{E}(\Sigma,\varphi_{\ig\Sigma},\lambda)\oplus \mathcal{K}_{\Lambda^1}(\lambda+1)
&
\textrm{if}\,
\lambda\in (-1,0).
\end{cases}
\end{align*}
\end{proposition}
\begin{proof}
By Remark \ref{remark-char-asd} and Lemma \ref{no-log}
any $\eta\in\mathcal{K}_{\mathrm{ASD}}(\lambda)$ is of the form
\begin{align*}
\eta = r^{\lambda}
(r^3 dr\wedge\alpha + r^4 (-*\alpha))
\end{align*}
with 
\begin{align}
d\alpha = - (\lambda+4) *\alpha.
\label{equation on link}
\end{align}

The $G_2$-structure $\varphi_{\ig\Sigma}$ induces the decomposition $\Omega^3(\Sigma) = \Omega^3_1(\Sigma) \oplus \Omega^3_7(\Sigma)\oplus \Omega^3_{27}(\Sigma)$ of 3-forms which we can use to write 
\begin{align}
\label{type-decomp-alpha}
\alpha = 
\alpha_1 + \alpha_7 + \alpha_{27}
=
f \varphi_{\ig\Sigma} + X\lrcorner *\varphi_{\ig\Sigma} + \zeta,
\end{align}
where $f$ is a function and $X$ a vector field on $\Sigma$,
 and $\zeta\in \Omega^3_{27}(\Sigma)$. 
Our goal is to show that if $\lambda\in(-4,1)$ then $f$ and $X$ have to vanish. 
The main idea to prove this is to interpret the condition $d\eta = 0$ on the cone as an equation involving the Laplace operator and use the fact that the Laplace operator preserves the type decomposition on the cone with respect to the Spin(7)-structure. More specifically, because $\lambda\neq -4$ we can write
$\eta = d\beta$,
where $\beta = \frac{1}{\lambda+4}r^{\lambda+4}\alpha$.
The fact that $\eta$ is a closed anti-self-dual 4-form on $C(\Sigma)$ homogeneous of rate $\lambda$ implies 
that $\beta$ is a harmonic 3-form on $C(\Sigma)$ homogeneous of rate $\lambda+1$. Indeed 
by \eqref{formulas-cone} $\beta$ is co-closed because $\alpha$ is co-closed,
and hence
\begin{align*}
\Delta\beta
=
dd^* \beta + d^*d\beta 
=
d^*\eta = 0.
\end{align*}
Next we
relate the type decomposition \eqref{type-decomp-alpha} of $\alpha$ with respect to the $G_2$-structure on the link to the type decomposition of $\beta=\beta_8+\beta_{48}$ with respect to the Spin(7)-structure on the cone. Because the decomposition is linear in $\alpha$, we compute the contributions of $\alpha_1, \alpha_7, \alpha_{27}$ separately.
We have
\begin{align}
\frac{r^{\lambda+4}}{\lambda+4} \alpha_1
=
\frac{r^{\lambda+4}}{\lambda+4} f\varphi_{\ig\Sigma}
=
\frac{r^{\lambda+1}}{\lambda+4} (f \partial_r)\lrcorner\psi_{\ig C},
\label{contr-a1}
\end{align}
and $\alpha_1$ only contributes to $\beta_8$.

Write $\pi_8(r^{\lambda+4}\alpha_7)=Y\lrcorner \psi$ for some vector field on $C(\Sigma)$ to be determined.
By the computation \ref{contr-a1} $Y$ does not contain a $\partial_r$ component. 
Therefore, we get
\begin{align*}
\pi_8(r^{\lambda+4}\alpha_7) &= -r^3 dr\wedge(Y\lrcorner \varphi_{\ig\Sigma})+r^4 Y\lrcorner *\varphi_{\ig\Sigma},
\\
\pi_{48}(r^{\lambda+4}\alpha_7) 
&= r^{\lambda+4} X\lrcorner \psi_{\ig C}
-
\pi_8(r^{\lambda+4}\alpha_7)
=
dr\wedge(r^3 Y\lrcorner\varphi_{\ig\Sigma})+(r^{\lambda+4}X-r^4Y)\lrcorner *\varphi_{\ig\Sigma}.
\end{align*}
The characterising equation $\pi_{48}(r^{\lambda+4}\alpha_7)\wedge\psi=0$ for forms of type 48 leads to the equation
\begin{gather*}
dr\wedge*\varphi_{\ig\Sigma}\wedge(r^7 Y \lrcorner \varphi_{\ig\Sigma})
+
dr\wedge\varphi_{\ig\Sigma}\wedge((r^{\lambda+7}X-r^7Y)\lrcorner *\varphi_{\ig\Sigma})
\\
+
*\varphi_{\ig\Sigma}\wedge((r^{\lambda+8}X-r^8Y)\lrcorner *\varphi_{\ig\Sigma})
=
0,
\end{gather*}
which is equivalent to the system
\begin{align*}
*\varphi_{\ig\Sigma}\wedge((r^{\lambda+8}X-r^8Y)\lrcorner *\varphi_{\ig\Sigma})
=
0,
\\
*\varphi_{\ig\Sigma}\wedge(r^7 Y \lrcorner \varphi_{\ig\Sigma})
+
\varphi_{\ig\Sigma}\wedge((r^{\lambda+7}X-r^7Y)\lrcorner *\varphi_{\ig\Sigma})
=
0.
\end{align*}
By \cite[Lemma 2.2.3]{Spiro-PhD} the first equation
is always true while the second equation is equivalent to
\begin{align*}
3*Y^{\flat}-4*(r^{\lambda} X- Y)^{\flat}
=
0,
\end{align*}
which gives
$Y = \frac{4}{7}r^{\lambda}X$.

Because $\zeta\wedge\varphi = 0$ and $\zeta\wedge*\varphi = 0$, we immediately get $\zeta\wedge\psi = 0$. Hence, $\alpha_{27}$ only contributes to $\beta_{48}$.

Adding the individual contributions gives
$\beta_8 = Z\lrcorner\psi$ with
\begin{align*}
Z 
=
\frac{r^{\lambda}}{\lambda+4}
\left(
r f \partial_r + \frac{4}{7} X
\right).
\end{align*}
Because the Laplace operator preserves type decompositions, we get $\Delta\beta_8 = 0$
and that $Z^{\flat} = \frac{r^{\lambda+1}}{\lambda+4}(f dr + \frac{4}{7} r X^{\flat})$ is a homogeneous harmonic 1-form of rate $\lambda+1$. Here $Z^{\flat}$ is dual to $Z$ on the cone and $X^{\flat}$ is dual to $X$ on $\Sigma$.
If $\lambda \leq -1$, then $Z^{\flat}$ must vanish by Lemma \ref{critical rates for Laplace on 1-forms}, and the statement follows in this case.

We are left to finish the proof for $\lambda\in(-1,0)$.
In this case $Z^{\flat}$ is closed and co-closed by Lemma \ref{critical rates for Laplace on 1-forms}.
If $\alpha'=f \varphi + X\lrcorner *\varphi + \zeta'$ is another solution of  
\eqref{equation on link}, then $\zeta'-\zeta\in\mathcal{E}(\Sigma,\lambda)$.
This proves
\begin{align}
\dim (\mathcal{K}_{\mathrm{ASD}}(\lambda)/\mathcal{E}(\Sigma,\varphi_{\ig\Sigma},\lambda))
\leq \dim \mathcal{K}_{\Lambda^1}(\lambda+1).
\label{dim-comp}
\end{align}
By Lemma \ref{0-1-Dirac}
$\beta_8=Z^{\flat}\lrcorner\psi$ is in the kernel of $\slashed{D}_{-}$. Therefore,
$d\beta_8$ lies in $\mathcal{K}_{\mathrm{ASD}}(\lambda)$.
The equation $d\beta_8 = \mathcal{L}_{Z} \psi_{\ig C}$
and Lemma \ref{Killing-fields-cone} show that $d\beta_8$ is non-zero if $Z$ is non-zero. This proves the reverse inequality of \eqref{dim-comp}.
\end{proof}

We can now summarise our results as

\begin{theorem}
\label{Thm-precise-formulation}
Let $C:=(C(\Sigma),\psi_{\ig C})$ be a Spin(7)-cone, which is not isometric to Euclidean space, and $(M,\psi)$ an
AC Spin(7)-manifold asymptotic to $C$ with rate $\nu$.
Suppose $\nu \in (-4,0)$ and that $\nu+1$ and $\nu$ are non-critical rates for the Laplace-operator on $C$.

Then the moduli space $\mathcal{M}_{\nu}=\mathcal{X}_{\nu}/\mathcal{D}_{\nu+1}$ of all torsion-free AC Spin(7)-structures on $M$ asymptotic to $C$ with rate $\nu$ is an orbifold of dimension
\begin{align*}
\dim \mathcal{M}_{\nu}
=
\dim (\mathcal{H}^4_{-})_{L^2}
+
\dim \im\Upsilon^4
+
\sum_{\lambda \in \mathcal{D}(d+d^*)
\cap (-4,\nu)}
\dim \mathcal{E}(\Sigma,\varphi_{\ig\Sigma},\lambda)
,
\end{align*}
where 
\begin{align*}
\mathcal{E}(\Sigma,\varphi_{\ig\Sigma},\lambda)
:=
\{
\zeta\in\Omega^3_{27}(\Sigma)
|\
d\zeta
=
-(\lambda+4) *\zeta
\},
\end{align*}
and $\Upsilon^4$ is the restriction map \eqref{restriction map}.

Furthermore, if the stabiliser of $\psi$ in $\mathcal{D}_{\nu+1}$ is trivial or acts trivially in a neighbourhood of $\psi$ in $\mathcal{X}_{\nu}$, the moduli space $\mathcal{M}_{\nu}$ is smooth in a neighbourhood of $\psi \mathcal{D}_{\nu+1}$. In particular, this is the case if 
with respect to the type decomposition given by $\psi$ the projection map
\begin{align*}
(\mathcal{H}^4_{35})_{\nu} \rightarrow H^4(M,\mathbb{R})
\end{align*}
is injective. 
\end{theorem}
\begin{proof}
The statement for $\nu\in(-4,-1]$
follows from Propositions \ref{prop-smoothness-of-moduli}, \ref{first computation of tangent space} and \ref{prop-reformulation}.

It is left to extend the statement to rates $\nu\in(-1,0)$ if $(\Sigma,\varphi_{\ig \Sigma})$ is not the round 7-sphere.
By 
Proposition \ref{no-Killing} the exterior derivative is injective on 
$\ker(\slashed{D}_{-})_{\nu+1}$. With Lemma \ref{Injectivity of negative Dirac-operator}, Theorem \ref{index change for elliptic operators} and Lemma \ref{0-1-Dirac}
we get
\begin{align*}
&\dim d \ker(\slashed{D}_{-})_{\nu+1} 
= \dim \ker(\slashed{D}_{-})_{\nu+1}
=
\sum_{\lambda \in \mathcal{D}(\slashed{D}_{-})
\cap (0,\nu+1)}
\ker (\slashed{D}_{-})_{\lambda+\varepsilon}
-
\ker (\slashed{D}_{-})_{\lambda-\varepsilon}
\\
&=
\sum_{\lambda \in \mathcal{D}(\slashed{D}_{-})
\cap (0,\nu+1)}
\ind (\slashed{D}_{-})_{\lambda+\varepsilon}
-
\ind (\slashed{D}_{-})_{\lambda-\varepsilon}
=
\sum_{\lambda \in \mathcal{D}(d+d^*)
\cap (-1,\nu)}
\dim \mathcal{K}_{\Lambda^1}(\lambda+1).
\end{align*}
With Proposition \ref{prop-reformulation} we get
\begin{align*}
&\dim \mathcal{M}_{\nu}
=
\dim (\mathcal{H}^4_{35})_{\nu}-\dim d \ker (\slashed{D}_{-})_{\nu+1}
\\
&=
\dim (\mathcal{H}^4_{-})_{L^2}
+
\dim \im\Upsilon^4
\\
&+
\sum_{\lambda \in \mathcal{D}(d+d^*)
\cap (-4,-1]}
\dim \mathcal{K}_{\mathrm{ASD}}(\lambda)
+
\sum_{\lambda \in \mathcal{D}(d+d^*)
\cap (-1,\nu)}
\dim \mathcal{K}_{\mathrm{ASD}}(\lambda)
-
\dim \mathcal{K}_{\Lambda^1}(\lambda+1)
\\
&=
\dim (\mathcal{H}^4_{-})_{L^2}
+
\dim \im\Upsilon^4
+
\sum_{\lambda \in \mathcal{D}(d+d^*)
\cap (-4,\nu)}
\dim \mathcal{E}(\Sigma,\varphi_{\ig\Sigma},\lambda).
\end{align*}

\end{proof}

In the compact case the projection of the moduli space to $H^4(M)$ is an immersion \cite[Theorem 10.7.1]{big-joyce}. In the AC case we can only prove this under the restriction that there are no critical rates in the interval $(-4,\nu)$, and in particular all Spin(7)-structures in $\mathcal{M}_{\nu}$ decay with rate $-4$.

\begin{lemma}
\label{lemma-immersion}
Let $\nu\in(-4,-0)$ and suppose there is no critical rate in the interval $(-4,\nu]$. Then $\mathcal{M}_{\nu}$ is a smooth manifold and the map
\begin{gather*}
\pi\colon \mathcal{M}_{\nu} \rightarrow H^4(M),
\\
\tilde{\psi}\mathcal{D}_{\nu+1}
\mapsto
[\tilde{\psi}]
\end{gather*}
is an immersion.
\end{lemma}
\begin{proof}
By the assumption, equation \eqref{tangent-space} and Proposition \ref{ASD kernel change at rates > -4}
we have 
$(\mathcal{H}^4_{35})_{\nu}
=
(\mathcal{H}^4_{35})_{-4+\varepsilon}$
for some arbitrarily small $\varepsilon > 0$.
We need to show that the projection 
\begin{align*}
(\mathcal{H}^4_{35})_{-4+\varepsilon}\rightarrow H^4(M).
\end{align*}
is injective.
Assume that 
$\gamma
\in 
(\mathcal{H}^4_{35})_{-4+\varepsilon}$
is exact. 
By Corollary \ref{ASD change at -4} there exists $\beta\in\im\Upsilon^4$
and $\gamma_{-}\in\mathcal{C}^{\infty}_{-4-\varepsilon}(\Lambda^4 T^* M)$ such that
\begin{align*}
\gamma
=
\chi (r^{-1} dr\wedge(-*_{\ig\Sigma}\beta)+\beta) + \gamma_{-},
\end{align*}
and $\Upsilon^4([\gamma])=[\beta]$. By assumption $[\gamma]=0$ and thus $\beta = 0$ and $\gamma=\gamma_{-}\in \mathcal{H}^4_{L^2}$.
By Proposition \ref{Hodge Theorem} $[\gamma]=0$ implies $\gamma=0$. 
Therefore, the linearisation is injective and $\pi$ is an immersion.\end{proof}

\section{Example: the Bryant--Salamon metric}

\label{Bryant-Salamon section}

The Bryant--Salamon metric on $\mathbf{S}_{+}(S^4)$ is a cohomogeneity one AC Spin(7) holonomy metric asymptotic
with rate $-10/3$ to the cone over the ``squashed'' 7-sphere. 
In this section we will compute the contributions $\mathcal{E}(\Sigma,\varphi_{\ig \Sigma},\lambda)$ to the moduli space following Alexandrov--Semmelmann \cite{alexandrov2012} by using the fact that the squashed 7-sphere can be understood as a standard homogeneous space. 
We will briefly describe their method. Let $G/H$ be a reductive 7-dimensional homogeneous space with reductive decomposition $\mathfrak{g}=\mathfrak{h}\oplus \mathfrak{m}$, where $\mathfrak{g}$ and $\mathfrak{h}$ denote the Lie algebras of $G$ and $H$, respectively. 
Denote by $\bar{\nabla}$ the canonical homogeneous connection on $G/H$.
We say that $G/H$ is a \textit{naturally} reductive homogeneous space if
the torsion tensor $T_{\bar{\nabla}}$ of $\bar{\nabla}$ is an alternating tensor, i.e. a 3-form.
We are interested in the situation where $G/H$ is equipped with a $G$-invariant nearly parallel $\mathrm{G}_2$-structure $\varphi_{\ig\Sigma}$ such that $\varphi_{\ig\Sigma}=\frac{2}{3}T_{\bar{\nabla}}$ (see \cite[Lemma 7.1]{alexandrov2012}).
This allows Alexandrov--Semmelmann to relate the 
 Laplacian
\begin{align*}
\bar{\Delta}
=
\bar{\nabla}^*\bar{\nabla}+q(\bar{R})
\end{align*}
to the  Laplacian with respect to the Levi-Civita connection.
If $\zeta\in\mathcal{E}(\Sigma,\varphi_{\ig \Sigma},\lambda)$, then similarly as in Proposition \ref{Laplace-formulation-ana-contr}  in the eigenproblem \eqref{Laplace-inf-defo}
we get a shift of the eigenvalue
with \cite[Proposition 5.3]{alexandrov2012}:
\begin{align}
\bar{\Delta}\zeta = \Delta \zeta + \frac{2}{3} * d\zeta
=
\underbrace{\big((\lambda+4)^2-\frac{2}{3}(\lambda+4)\big)}_{=: \mu} \zeta.
\label{homogeneous-eigenvalue}
\end{align}
To compute the action of $\bar{\Delta}$ we need to make another restriction:
we require that the Einstein metric induced by $\varphi_{\ig\Sigma}$ is \textit{standard}. 
This means that it is induced by a negative multiple $-c^2 B$ of the Killing form $B$ of $\mathfrak{g}$. The point of standard homogeneous spaces is that their curvature tensor with respect to the canonical homogeneous connection satisfies the same formula as the curvature tensor of symmetric spaces with respect to the Levi-Civita connection. Therefore, eigenproblems for $\bar{\Delta}$ can be solved with methods from representation theory. For a representation $\rho\colon H \rightarrow \mathrm{GL}(E)$
of $H$ denote the associated vector bundle by $E_{\rho}=G\times_{\rho}E$.
The left action of $G$ on $E_{\rho}$ induces a $G$-action $\ell\colon G \rightarrow \mathrm{GL}(\Gamma(E_{\rho}))$ on the space of sections of $E_{\rho}$. Then by \cite[Lemma 5.2]{NK} the action of $\bar{\Delta}$ is given by
\begin{align*}
\bar{\Delta}=-\frac{1}{c^2} \mathrm{Cas}^G_{\ell}.
\end{align*}
The Casimir operator acts on a $G$-representation $\gamma\colon G \rightarrow \mathrm{GL}(V)$ as
\begin{align*}
\mathrm{Cas}^G_{\gamma}
=
\sum_{i} (\gamma_{*} X_i)^2,
\end{align*}
where $\gamma_{*}$ denotes the induced action of the Lie algebra and $X_i$ is an orthonormal basis of $\mathfrak{g}$ with respect to $-B$.
The Peter--Weyl Theorem and the Frobenius reciprocity give an isomorphism
\begin{align}
L^2(E_{\rho})
=
\overline{\bigoplus_{\gamma}} V_{\gamma}\otimes \mathrm{Hom}_{H}(V_{\gamma},E),
\label{Pter-Weyl}
\end{align}
where $\gamma$ runs over all isomorphism classes of irreducible representations $V_{\gamma}$ of $G$.
A section of $E_{\rho}$ is the same as an $H$-invariant map $G\rightarrow E$.
Under this identification an element $v\otimes A \in V_{\gamma}\otimes \mathrm{Hom}_{H}(V_{\gamma},E)$ gives rise to the section $g\mapsto A(\gamma(g^{-1})v)$. On each component $V_{\gamma}\otimes \mathrm{Hom}_{H}(V_{\gamma},E)$ $\mathrm{Cas}^G_{\ell}$ then acts as $\mathrm{Cas}^G_{V_{\gamma}}$. Therefore, the eigenspace of $\bar{\Delta}$ for the eigenvalue $\mu$ is isomorphic to the sum of the spaces $V_{\gamma}\otimes \mathrm{Hom}_{H}(V_{\gamma},E)$ for which 
\begin{align}
\mathrm{Cas}^G_{V_{\gamma}} = -c^2 \mu.
\label{eigenvalue-to-find}
\end{align}

We can now compute $\mathcal{E}(\Sigma,\varphi_{\ig\Sigma},\lambda)$
in two steps. Set $E=\Lambda^3_{27} \mathfrak{m}$ and suppose that
in the orientation chosen by Alexandrov--Semmelmann $\mathcal{E}(\Sigma,\varphi_{\ig\Sigma},\lambda)$ is characterised 
by solutions of the equation $\bar{d}\zeta + \bar{\lambda} *\zeta = 0$,
where $\bar{d}= \mathrm{Alt}\circ \bar{\nabla}$ and $\bar{\lambda}$ is a constant related to $\lambda$, and each $\zeta\in\mathcal{E}(\Sigma,\varphi_{\ig\Sigma},\lambda)$ satisfies $\bar{\Delta}\zeta = \mu \zeta$.
First, using \eqref{eigenvalue-to-find} we determine all isomorphism classes of irreducible representations $V_{\gamma}$ of $G$
such that $\mathrm{Cas}^G_{V_{\gamma}} = -c^2 \mu$.
Secondly, having narrowed down the list of possible $V_{\gamma} \subset \mathcal{E}(\Sigma,\varphi_{\ig\Sigma},\lambda)$, we need to solve the equation $\bar{d}\zeta + \bar{\lambda} *\zeta = 0$.
All $\zeta$ in a subspace isomorphic to $V_{\gamma}$ solving this equation is equivalent to the existence of $A\in \mathrm{Hom}_{H}(V_{\gamma},E)$ such that (see \cite[Equation (7.42)]{alexandrov2012})
\begin{align}
\label{equ-for-hom}
\sum_{1 \leq i_1 < \cdots < i_4 \leq 7}
\sum_{j=1}^4
(-1)^j A(e_{i_j} \cdot v)(e_{i_1}, \dots , \hat{e}_{i_j}, \dots, e_{i_4}) e^{i_1\dots i_4}
+ \bar{\lambda} *A(v)=0
\end{align} 
for all $v\in V_{\gamma}$, where $e_1, \dots, e_7$ is a basis of $\mathfrak{m}$.
With respect to the identification \eqref{Pter-Weyl} $V_{\gamma}$ is then embedded into $ \mathcal{E}(\Sigma,\varphi_{\ig\Sigma},\lambda)$ via
\begin{align*}
V_{\gamma} 
\rightarrow V_{\gamma}\otimes \mathrm{Hom}_{H}(V_{\gamma},E)
\subset L^2(E_{\rho}),
\quad 
v \mapsto v\otimes A.
\end{align*}

Let us now apply this theory to compute the spaces $\mathcal{E}(\Sigma,\varphi_{\ig\Sigma},\lambda)$ for the Bryant--Salamon metric. The ``squashed'' nearly parallel $\mathrm{G}_2$-structure on $S^7$ is not naturally reductive if we write $S^7=\Sp(2)/\Sp(1)$.
However, it is both naturally reductive and standard 
if we write $S^7=\frac{\Sp(2)\times \Sp(1)}{\Sp(1)_u\times \Sp(1)_d}$ \cite[Example 8.2]{alexandrov2012},
where 
\begin{align*}
\Sp(1)_u
=
\bigg\{
\left(
\begin{bmatrix}
a & 0
\\
0 & 1
\end{bmatrix}
,1\right)
:
a\in \Sp(1)
\bigg\}
,
\quad
\Sp(1)_d
=
\bigg\{
\left(
\begin{bmatrix}
1 & 0
\\
0 & a
\end{bmatrix}
,a\right)
:
a\in \Sp(1)
\bigg\}.
\end{align*}
This description leads to a nearly parallel $\mathrm{G}_2$-structure 
satisfying $d\varphi = \tau_0 *\varphi$
with scalar curvature $\frac{21}{8}\tau_0^2$, where $\tau_0 = \frac{12}{\sqrt{5}}$. This means that we have to rescaled our original choice of $\varphi$ by $\kappa^3$ and our original metric $g_{\ig \Sigma}$ by $\kappa^2$, where $\kappa$ is given by the equation $\tau_0=\frac{4}{\kappa}$. The eigenproblem \eqref{homogeneous-eigenvalue} is replaced by
\begin{align*}
\bar{\Delta}\zeta = \kappa^{-2} \mu \zeta.
\end{align*}
By \cite[Lemma 7.1]{alexandrov2012} $c$ and $\tau_0$ are related by $c^2 = \frac{6}{5\tau_0^2}$.
In the light of equation \eqref{eigenvalue-to-find}
we need to determine those irreducible representations $V_{\gamma}$ of $G=\Sp(2)\times \Sp(1)$ for which 
\begin{align}
\mathrm{Cas}^G_{\gamma}=-c^2 \kappa^{-2} \mu = - \frac{6}{5 \tau_0^2} \frac{\tau_0^2}{16}\mu = -\frac{3}{40}\mu.
\label{scaled-eigenvalue}
\end{align}
Irreducible representations of $\Sp(2)$
are indexed by their highest weight $\gamma = (k_1,k_2), k_1 \geq k_2\geq 0$, and irreducible representations of $\Sp(1)$ are indexed by their highest weight $\gamma = l, l\geq 0$. The Casimir operator is explicitly given by (see \cite[p. 737]{alexandrov2012})
\begin{align}
\label{Casimir-explicit}
\mathrm{Cas}^{\Sp(2)\times \Sp(1)}_{V(k_1,k_2)\otimes V(l)}
=
-\frac{1}{12}(4 k_1 + k_1^2 +2 k_2 + k_2^2) -\frac{1}{8}(2l+l^2).
\end{align}
The eigenvalue $\mu$ in equation \eqref{homogeneous-eigenvalue} takes values in $(-\frac{1}{9},\frac{40}{3})$ for $\lambda\in(-4,0)$.
Therefore, by \eqref{scaled-eigenvalue}
we need to determine all $V(k_1,k_2,l):= V(k_1,k_2)\otimes V(l)$ such that 
\begin{align*}
\mathrm{Cas}^{\Sp(2)\times \Sp(1)}_{V(k_1,k_2)\otimes V(l)}\in (-1,0].
\end{align*} 
Using formula \eqref{Casimir-explicit}
we find that there are four possibilities: $(k_1,k_2,l)=(1,1,0)$, $(1,0,0)$, $(0,0,1)$, $(0,0,0)$ with the Casimir operator equal to $-\frac{2}{3}, -\frac{5}{12}, -\frac{3}{8}$, $0$, respectively. This will lead to the eigenvalues $\kappa^{-2}\mu = \frac{576}{25},\frac{72}{5},\frac{324}{25}$,
$0$, respectively.

Next we need to determine the corresponding Hom-spaces.
If we denote the standard representations of $\Sp(1)_u$ and $\Sp(1)_d$ by $U$ and $D$, then all irreducible representations of $\Sp(1)_u \times \Sp(1)_d$ can be written via the symmetric powers as $S^k U S^l W$ (omitting the tensor product sign and complexification sign).
Then 
\begin{gather*}
\Lambda^3_{27} \mathfrak{m}^* 
\cong
S^2 U S^2 D \oplus U S^3 D \oplus U D \oplus S^4 D \oplus \mathbb{C},
\\
V(1,1,0) \cong \Lambda^2_0 (\mathbb{C}^4)^* \cong UD \oplus \mathbb{C},
\quad
V(1,0,0)\cong \mathbb{C}^4 \cong U \oplus D, 
\\
V(0,0,1) \cong D,
\quad
V(0,0,0) \cong \mathbb{C}.
\end{gather*}
$V(1,0,0)$ and $V(0,0,1)$ do not have common subrepresentations with $\Lambda^3_{27} \mathfrak{m}$, and therefore do not lead to any solutions.

$V(0,0,0)$ and $\Lambda^3_{27}\mathfrak{m}^*$ have the trivial representation $\mathbb{C}$ as a common component. We have $\mu=0$, $\lambda=-10/3$,
and $\zeta\in\mathcal{E}(\Sigma,\varphi_{\ig\Sigma},\lambda)$ satisfy
$d\zeta=-\frac{2}{3}*\zeta$. Then we have $\bar{d}\zeta= d\zeta + \frac{2}{3}*\zeta = 0$ by \cite[Lemma 5.2]{alexandrov2012}.
Therefore, we want to solve equation \ref{equ-for-hom} with $\bar{\lambda}=0$. This is trivially satisfied for the trivial representation.
We get $\mathcal{E}(\Sigma,\varphi_{\ig\Sigma},-10/3)\cong \mathbb{R}$.
The Bryant--Salamon metric has decay rate $-10/3$ and this is exactly the deformation given by rescaling as in Remark \ref{remark-scaling}.

We are left to consider $V(1,1,0)$. Then $\mu = \frac{64}{5}$ and $\lambda=-\tfrac{-55+\sqrt{2905}}{15}$. Thus, at the scale of scalar curvature 42
and in our chosen orientation we want to solve $d\zeta = -\tfrac{5+\sqrt{2905}}{15}*\zeta$. Again by \cite[Lemma 5.2]{alexandrov2012} 
this is equivalent to $\bar{d}\zeta = d\zeta + \frac{2}{3}*\zeta = \frac{5-\sqrt{2905}}{15}*\zeta$, and at the scale with scalar curvature $\frac{21}{8}\tau_0^2$ and in the orientation chosen in \cite{alexandrov2012} we want to solve 
\begin{align}
\big(\bar{d}+\tfrac{\sqrt{5}-\sqrt{581}}{5} *\big)\zeta =0.
\label{what-we-want-to-solve}
\end{align}

The common $\Sp(1)_u\times \Sp(1)_d$-subrepresentation of $\Lambda^3_{27}\mathfrak{m}$
and $V(1,1,0)$ are $UD$ and $\mathbb{C}$. Alexandrov--Semmelmann show that $\mathrm{Hom}_H(UD, \Lambda^3_{27}\mathfrak{m})$ is 1-dimensional. Furthermore, they show that   $UD$ can be identified with a submodule of $\mathfrak{sp}(2)$ and that a generator $A$ of $\mathrm{Hom}_H(UD, \Lambda^3_{27}\mathfrak{m})$ satisfies with repsect to an orthonormal frame $e_1, \dots, e_7$ of $\mathfrak{m}$
\begin{align*}
A(e_1 \cdot e_4)
&=
-\tfrac{2}{\sqrt{5}}
\big(
3 e^{467} + e^{137}+e^{126}+e^{234}
\big),
\\
A(e_2 \cdot e_4)
&=
-\tfrac{2}{\sqrt{5}}
\big(
-3 e^{457} + e^{237}-e^{125}-e^{134}
\big),
\\
A(e_3 \cdot e_4)
&=
-\tfrac{2}{\sqrt{5}}
\big(
3e^{456}-e^{236}-e^{135}+e^{124}
\big) 
,
\\
A(e_4 \cdot e_4) 
&= 0,
\\
A(e_4)
&=
-3 e^{567} - e^{235}+e^{136}-e^{127}.
\end{align*}
Therefore, if in \eqref{equ-for-hom} we set $\bar{\lambda}=\tfrac{\sqrt{5}-\sqrt{581}}{5}$  and $v=e_4$, the coefficient of $e^{1234}$ on the left-hand side is $3\frac{2}{\sqrt{5}}-3 \bar{\lambda}=3\frac{\sqrt{5}+\sqrt{581}}{5}\neq 0$.
Therefore, the common subrepresentation $UD$ does not lead to any solutions. For the trivial representation formula \eqref{equ-for-hom}
simplifies to
\begin{align*}
\bar{\lambda}* A(v) = 0
\end{align*}
for all $v\in\mathbb{C}$. This implies $A=0$. Again we get no solution. Therefore, $V(1,1,0)$ does not lead to any infinitesimal Spin(7)-deformations.

The computations in this section allow us
to determine the dimension of the moduli space with Theorem \ref{Thm-precise-formulation}.

\begin{corollary}
The Bryant--Salamon AC Spin(7) holonomy metric on $\mathbf{S}_{+}(S^4)$
is locally rigid, modulo scaling, as a torsion-free AC Spin(7)-structure on $\mathbf{S}_{+}(S^4)$ asymptotic to the cone over the ``squashed'' 7-sphere, up to any rate $\nu < 0$.
\end{corollary}
\begin{proof}
So far we have shown that the spaces $\mathcal{E}(\Sigma,\varphi_{\ig\Sigma},\lambda)$, $\lambda\in (-4,0)$, vanish unless $\lambda = -10/3$.   $\mathcal{E}(\Sigma,\varphi_{\ig\Sigma},-10/3)$ is 1-dimensional and the corresponding deformation is the scaling described in Remark \ref{remark-scaling}. By Proposition \ref{Hodge Theorem}
and the long exact sequence \eqref{long-exact-sequence} we have 
\begin{align*}
\mathcal{H}^4_{L^2}\cong \mathcal{I}^4(H^4_{\mathrm{cs}}(\mathbf{S}_{+}(S^4),\mathbb{R}))\cong  H^4(\mathbf{S}_{+}(S^4),\mathbb{R}) \cong H^4(S^4,\mathbb{R})\cong \mathbb{R}.
\end{align*}
Cveti\v{c}--L\"u--Pope \cite[Section 5.2]{cvetivc2001brane} have constructed a square-integrable 4-form on $\mathbf{S}_{+}(S^4)$, which is harmonic with respect to the Bryant--Salamon metric and has the same duality as the Spin(7) 4-form (which is anti-self-dual in their convention and self-dual in our convention). Therefore, $(\mathcal{H}^4_{-})_{L^2}=\{0\}$. Finally we have $\im \Upsilon^4 \subset H^4(S^7,\mathbb{R}) = \{0\}$. By Theorem \ref{Thm-precise-formulation}
the moduli space $\mathcal{M}_{\nu}$
is 1-dimensional for any $\nu \in (-10/3,0)$.
\end{proof}

\bibliographystyle{amsalpha}

\bibliography{bibliography}

\end{document}